%% file: main.tex
\documentclass[12pt, reqno, a4paper,oneside]{amsart}
\usepackage{graphicx}
\usepackage{subfig}
\usepackage{color}
\usepackage{xcolor}
\usepackage[top=2.5cm,bottom=2.0cm,left=2.0cm,right=2.0cm]{geometry}
\usepackage{amsfonts, amsmath, amssymb, amsbsy, amsthm}
\usepackage{environments}
\usepackage{mathrsfs}
\usepackage{listings}
\usepackage{algpseudocode}
\usepackage{algorithm, algorithmicx}
\numberwithin{theorem}{section}
\numberwithin{equation}{section}

\input{notation}
\def\a{{\rm a}}

\def\bqce{{\rm bqce}}
\def\bqhoce{{\rm bqhoce}}

\def\err{{\mathrm{err}}}
\def\hoc{{\rm hoc}}

\definecolor{yscol}{HTML}{6622AA}
\definecolor{jlcol}{rgb}{0.2, 0.6, 0.2}
\definecolor{hwcol}{rgb}{0.0, 0, 1.0}

\def\T{\mathcal{T}}
\def\R{\mathbb{R}}

\definecolor{yscol}{HTML}{6622AA}

\title[Blended Atomistic to Higher-Order Continuum Coupling]{Formulation and Analysis of Blended Atomistic to Higher-Order Continuum Coupling Methods for Crystalline Defects}

\author{Junfeng Lu}
\address{Junfeng Lu\\
	School of Mathematics\\
	Sichuan University\\
	No.24 Yihuan Road\\
	Chengdu 610065, 
	China
}
\email{lujunfeng1@stu.scu.edu.cn}

\author{Hao Wang}
\address{Hao Wang\\
	School of Mathematics\\
	Sichuan University\\
	No.24 Yihuan Road\\
	Chengdu 610065, 
	China
}
\email{wangh@scu.edu.cn}

\author{Yangshuai Wang}
\address{Yangshuai Wang\\
	Department of Mathematics\\
	 National University of Singapore\\
	  10 Lower Kent Ridge Road\\
	   119076, Singapore
}
\email{yswang@nus.edu.sg}

\date{\today}


\begin{document}

\maketitle

\input{notation}
\begin{abstract}
	Concurrent multiscale methods play an important role in modeling and simulating materials with defects, aiming to achieve the balance between accuracy and efficiency. Atomistic-to-continuum (a/c) coupling methods, a typical class of concurrent multiscale methods, link atomic-scale simulations with continuum mechanics. Existing a/c methods adopt the classic {\it second-order} Cauchy-Born approximation as the continuum mechanics model. In this work, we employ a {\it higher-order} Cauchy-Born model to study the potential accuracy improvement of the coupling scheme. In particular, we develop an energy-based blended atomistic to higher-order continuum method and present a rigorous {\it a priori} error analysis. We show that the overall accuracy of the energy-based blended method is not actually improved due the coupling interface error which is of lower order and may not be improved.  On the contrast, higher order accuracy is achieved by the force-based blended atomistic to higher-order continuum method. Our theoretical results are demonstrated by a detailed numerical study. 
\end{abstract}

\section{Introduction}
\label{sec:intro}

Crystalline defects play a crucial role in determining the mechanical and functional properties of materials. Understanding their behavior is essential for unraveling material mechanisms and optimizing performance. Fully molecular mechanics simulations can capture these effects with high accuracy but are computationally prohibitive for large-scale systems~\cite{olson2023elastic, braun2022higher, fu2023adaptive}. To address this, atomistic-to-continuum (a/c) coupling methods (also known as the quasicontinuum methods particular in the engineering community) integrate the molecular mechanics or atomistic models around defect cores (also termed as the atomistic region) with continuum models in the elastic far-field (also termed as the continuum region). Such methods significantly reduce the computational cost but maintain the accuracy~\cite{1999_VS_RM_ETadmor_MOrtiz_AFEM_QC_JMPS,2002_PL_QCL_1D_MATHCOMP,2002_XB_LB_PL_From_MM_to_CM_ARMA,2009_Miller_Tadmor_Unified_Framework_Benchmark_MSMSE,2018_DO_XL_CO_BK_Force_Based_AC_Complex_Lattices_NUMMATH, fu2025meshac, wang2025posteriori, chen2022qm, wang2021posteriori, chen2019adaptive, wang2023adaptive}. We refer to~\cite{2013_ML_CO_AC_Coupling_ACTANUM} for a comprehensive overview and benchmark of a/c methods.

Based on the variational principles, a/c coupling methods can be broadly classified into energy-based and force-based approaches. Energy-based methods construct a hybrid energy functional by combining atomistic and continuum energy formulations, where the solution corresponds to the local minimum of this total hybrid energy~\cite{2002_PL_QCL_1D_MATHCOMP, 2007_PL_QCL_2D_SIAMNUM, 2008_CO_ES_QCM_1D_M2NA}. In contrast, force-based methods directly couple the atomistic and continuum forces to formulate a non-conservative hybrid force field and solve the problem through a system of force balance equations~\cite{PM_Error_QCF_2008,2008_MD_ML_Ana_Force_Based_QC_M2NA,2011_CM_CO_ES_Ana_QCF_SAC_1D_NUMMATH}. Each category can be further divided into sharp-interfaced and blended types, depending on how the transition between the atomistic and continuum regions is handled. Sharp-interfaced methods employ a reconstructed site energy at the atomistic/continuum interface, typically spanning only a few layers of atoms. Blended methods, on the other hand, introduce a longer transition region where the atomistic and continuum descriptions are gradually mixed. Although sharp-interfaced methods, such as the geometric reconstruction-based atomistic/continuum (GRAC) method~\cite{2012_CO_LZ_GRAC_Construction_SIAMNUM}, possess the advantage of faster transition of scales, the practical applications of such methods are restricted due to the complexity of the reconstruction of the interface energy. Blended methods, in contrast, are widely adopted for their simplicity of formulation. 

Blended methods have been widely used in multiscale simulations for quite some time~\cite{2004_BD_GR_IJNME,2004_SX_TB_Bridging_Domain_CMAME,2005_ST_TH_WKL_Bridging_Sacle_IJNME,2006_WL_HP_DQ_EK_HK_GW_Bridge_Scale_CMAME,2008_PB_BD_NE_TO_SP_Arlequin_AC_CM,2008_SP_BD_PB_TO_Anal_Coupling_Error_Arlequin_CMAME,2010_LC_SP_BD_TO_Ghost_Force_Arlequin_IJNME}, with different researchers adopting various terms to describe the same concept, such as "handshake region," "bridging region," or "overlap region." All of these terms are essentially referring to the concept of the blending region in atomistic-to-continuum coupling methods. For instance, in the Arlequin method~\cite{2004_BD_GR_IJNME,2008_PB_BD_NE_TO_SP_Arlequin_AC_CM,2008_SP_BD_PB_TO_Anal_Coupling_Error_Arlequin_CMAME,2010_LC_SP_BD_TO_Ghost_Force_Arlequin_IJNME}, atomistic and continuum models coexist within the overlap region, with their respective contributions being proportionally weighted by a blending function, and a Lagrange multiplier is then applied to ensure the displacement consistency between the two models. In contrast to the Arlequin method, the energy-based atomistic-to-continuum coupling method directly mixes the energies of the atomistic and continuum models~\cite{2011_BK_ML_BQCE_1D_SIMNUM,2013_ML_CO_BK_BQCE_CMAME,2016_XL_CO_AS_BK_BQC_Anal_2D_NUMMATH}, forming a unified model that minimizes displacement energy. This approach allows for the reduction of ghost forces at the interface by adjusting the length of the blending region. As for the force-based atomistic-to-continuum coupling method, due to the absence of $W^{1,2}$ stability~\cite{2010_MD_ML_CO_QCF_Stability_ARMA,2012_MD_CO_AS_QCF_Spectrum_MMS}, constructing a sufficiently long blending region ensures that the forces from the atomistic and continuum models are smoothly coupled, which allows the model to exhibit $W^{1,2}$ stability~\cite{2016_XL_CO_AS_BK_BQC_Anal_2D_NUMMATH,2012_XL_ML_CO_BQCF_Stab_1D_2D_MMS,2014_HL_ML_CO_AS_BQCF_Compt_CMAME,2012_JL_PM_Convergence_BQCF_3D_No_Defects_CPAM}.

We focus on developing and analyzing blended a/c coupling methods in the current work. The blended a/c methods have been proposed and analysed in a series of works~\cite{2011_BK_ML_BQCE_1D_SIMNUM,2013_ML_CO_BK_BQCE_CMAME,2016_XL_CO_AS_BK_BQC_Anal_2D_NUMMATH, 2012_XL_ML_CO_BQCF_Stab_1D_2D_MMS,2012_JL_PM_Convergence_BQCF_3D_No_Defects_CPAM,2014_HL_ML_CO_AS_BQCF_Compt_CMAME}. However, these methods are employ the classical second-order Cauchy-Born approximation ~\cite{2007_WE_PM_Cauchy_Born_Rule_and_The_Stability_ARMA, 2013_CO_FT_Cauchy_Born_ARMA} as the continuum model. This inherently limits the overall accuracy of the coupling scheme. Recent studies have introduced a few higher-order Cauchy-Born approximations ~\cite{2021_YW_LZ_HW_HOC_IMANUM, 2005_MA_MG_Higher_Order_Continuum_Model_MMS}. In particular, we prove in ~\cite{2021_YW_LZ_HW_HOC_IMANUM} that a particular constructed Cauchy-Born approximation can achieve fourth-order accuracy relative to the full atomistic model. This motivates the current study of the development and analysis of coupling strategies that integrate the atomistic model with a higher-order continuum model for the simulation of crystalline defects. By incorporating a higher-order continuum model, we intend to improve the overall accuracy of the method, particularly the consistency error of the coupling interface which is often the key limitation of the coupling scheme. 

In this work, we formulate an energy-based blended atomistic to higher-order continuum coupling method. We provide a rigorous {\it a priori} error analysis for such method and show that while the higher-order continuum model significantly enhances the accuracy within the continuum region, it does not reduce the coupling error at the coupling interface. A insightful mathematical explanation for this phenomenon is given by our analysis, revealing the fundamental limitation of the energy-based blending approach in this context. We also propose a force-based blended atomistic to higher-order continuum coupling method which avoid the essential obstacle of the energy-based scheme to improve the order of accuracy. We provide a systematic comparison and discussion of the convergence behavior for different blended methods and validate all aspects of our theoretical findings by comprehensive numerical experiments, demonstrating the (in-)effectiveness of the proposed methods.

We restrict ourselves to a one dimensional setting since the primary objective of this study is to clearly establish a formulation for blended atomistic to higher-order continuum coupling method and investigate the possibility of obtaining an improved accuracy compared with the classical a/c methods. However, we note that the error analysis presented here is general and is expected to extend to blended a/c coupling methods in higher dimensions. This study serves as a proof of concept, laying the groundwork for the extensions, which are further discussed in the conclusion and outlook.

\subsection{Outline}
\label{sec:sub: Outline}

The paper is structured as follows. In Section~\ref{sec: A-C Models}, we introduce the atomistic model along with its Cauchy-Born approximation. Section~\ref{sec: B-QHOCF and B-QHOCE} presents the formulation of both the energy-based and the force-based blended atomistic to continuum coupling method between the atomistic and the higher-order continuum model. In Section~\ref{sec: consistency}, we provide a rigorous {\it a priori} analysis for the energy-based blended coupling method together with a discussion for different blended a/c methods. We validate our theoretical findings through numerical experiments in Section~\ref{sec: Numerical experiments}. In Section~\ref{sec:conclusion}, we present the conclusions of this study and provide a detailed discussion on potential future research directions based on our findings.


\subsection{Notations}
\label{sec:sub: Natations}
We use the symbol $\langle\cdot,\cdot\rangle$ to denote an abstract duality
pairing between a Banach space and its dual. The symbol $|\cdot|$ normally
denotes the Euclidean or Frobenius norm, while $\|\cdot\|$ denotes an operator
norm. Directional derivatives are denoted by $\nabla \rho f := \nabla f\cdot\rho,\, \rho \in \mathbb{R}$. Letting the symbol $\langle \cdot,\cdot\rangle $
denote the duality pairing, the first and second variations of $\mathscr{F}$ at $u$ are defined as 
\begin{align}
	\langle\delta\mathscr{F}(u),v\rangle&:=\lim_{t\to 0}t^{-1}(\mathscr{F}(u+tv)-\mathscr{F}(u))\nonumber\\
	\text{and}\quad \langle\delta^2\mathscr{F}(u),v\rangle&:=\lim_{t\to 0}t^{-1}(\delta\mathscr{F}(u+tv)-\delta\mathscr{F}(u)).\nonumber
\end{align}

We write $|A| \lesssim B$ if there exists a constant $C$ such that $|A|\leq CB$, where $C$ may change from one line of an estimate to the next. When estimating rates of decay or convergence, $C$ will always remain independent of the system size, the configuration of the lattice and the the test functions. The dependence of $C$ will be normally clear from the context or stated explicitly.

\section{Preliminaries: The Atomistic Model and its Higher-order Cauchy-Born Approximations}
\label{sec: A-C Models}

In this section, we introduce the atomistic-to-continuum (a/c) coupling methods in one dimension. Extending these methods to higher dimensions is not straightforward, particularly due to the lack of rigorous analysis for higher-order continuum models in higher dimensions. Therefore, we focus on the one-dimensional case in this work. Section~\ref{sec:sub:atom} presents the atomistic model, which serves as the reference model. Section~\ref{sec:sub:hoc} provides an overview of the higher-order continuum models studied in~\cite{2021_YW_LZ_HW_HOC_IMANUM}.

\subsection{The atomistic model}
\label{sec:sub:atom}

We begin by introducing the atomistic or molecular mechanics model. Consider a periodic domain extracted from the infinite lattice $\mathbb{Z}$, consisting of $2N$ atoms or lattice points, denoted as $\Lambda :=\{-N+1, \ldots, N \}$. The computational domain is defined as $\Omega = [-N, N]$. Let $\T := \{T_j\}_{j=-N+1}^{N}$, where $T_j = [j-1, j]$, be the {\it canonical} finite element partition corresponding to $\Lambda$.

Let $ y_F:=F\mathbb{Z}$ be a uniformly deformed configuration under a macroscopic strain $F>0$. The space of $2N$-periodic displacements with central Dirichlet condition is then  defined as 
$$\mathscr{U}:=\Big\{ u: u({\ell+2N})=u(\ell) \;\text{for}\; \ell\in\Lambda, \,\text{and} \,u(0)=0\Big\},$$ 
and the set of {\it admissible} deformation is given by
$$\mathscr{Y}:=\left\{ y: y=y_F+u \;\text{for}\; u \in \mathscr{U}\right\}.$$

We consider an atomistic system with a finite interaction range $\mathscr{R}\subset\left\{1,2,...,r_{\rm cut}\right\}$ where $r_{\rm cut}>0$ is a cut-off radius. For a lattice function $v\in\mathscr{U}$, we define the finite difference operator 
$$
D_{\rho}v(\xi):=v(\xi+\rho)-v(\xi) \quad  \text{ for } \xi\in\mathbb{Z} \text{ and }\rho\in \mathscr{R}.		
$$

For an admissible deformation $y\in\mathscr{Y}$, the atomistic energy per period is given by
$$
\mathcal{E}^{\rm a}(y) := \sum_{\xi\in \Lambda}\sum_{\rho\in\mathscr{R}}\phi\big(D_{\rho}y(\xi)\big)=\sum_{\xi\in\Lambda}\sum_{\rho \in \mathscr{R}}\phi_{\rho}\big(D_{\rho}u(\xi)\big)=: \mathcal{E}^{\rm a}(u),
$$
where $\phi \in C^3((0,+\infty); \mathbb{R})$ is a pair potential, such as the Lennard-Jones interaction potential, and $\phi_{\rho}(r)$ is defined as $\phi_{\rho}(r) := \phi(r + F\rho)$. For analytical purposes, we assume the interatomic potential is sufficiently smooth. Specifically, we impose bounds on the derivatives of the interaction potential $\phi_{\rho}$: for $\rho\in\mathscr{R},1\leq j\leq 6$, we require

$$
M^{(j,s)}:=\sum_{\rho\in\mathscr{R}}\rho^{j+s}\sup_{g\in\mathbb{R}}\big|\phi_{\rho}^{(j)}(g)\big|<\infty.
$$

We equip the space of lattice functions $\mathscr{U}$ with the $H^1$-seminorm so that $$\mathscr{U}^{1,2} :=\left\{ u\in\mathscr{U}~\big|~\Vert\nabla u\Vert_{L^2(\Omega)}<+\infty\right\}.$$
Given a periodic dead load $f \in \mathscr{U}^{1,2}$, we define the work by the external force  
\begin{align}
	\langle f,u\rangle_\Lambda:=\sum_{\xi\in\Lambda}f(\xi)u(\xi).\label{Atomistic external force}
\end{align}
The atomistic problem we aim to solve is to find the local minimizers (denoted by $ \arg \min$) of the total energy
\begin{equation}
	\label{Atomistic variational problem}
	u^{\rm a}\in \arg \min\left\{\mathcal{E}^{\rm a}(u)-\langle f,u\rangle_\Lambda\ \big| \ u\in\mathscr{U}^{1,2}\right\}.
\end{equation}

We assume that the local minimizer \( u^{\rm a} \) satisfies the second-order optimality condition, given by  
\begin{equation}\label{atomistic model stability}
	\langle \delta^{2}\mathcal{E}^{\rm a}(u^{\rm a})v,v\rangle \ge c_{0} \Vert \nabla v \Vert^{2}_{L^{2}},  \quad \forall v\in \mathscr{U}^{1,2},
\end{equation}
for some constant \( c_{0} > 0 \). A solution \( u^{\rm a} \) satisfying this condition is referred to as a strongly stable local minimizer.

\subsection{The higher-order continuum model}
\label{sec:sub:hoc}

In this section, we introduce the higher order Cauchy-Born (or higher order continuum, HOC) approximation of the atomistic model \cite{2021_YW_LZ_HW_HOC_IMANUM, 2005_MA_MG_Higher_Order_Continuum_Model_MMS} and note that the classic Cauchy-Born (CB) approximation can be considered as a special case of the HOC. The idea of the higher-order Cauchy-Born approximations is that if the displacement $u$ (or its smooth interpolation) varies slowly then 
$$
D_{\rho}u(\xi)\approx\rho\nabla u(\xi')+\frac{\rho^3}{24}\nabla^3u(\xi'), \quad \text{where} \;\xi'=\frac{\xi+\xi+\rho}{2}.
$$
By substituting expansion into the atomistic energy and applying the step similar to Riemann sum \cite[Equation (3.9)]{2021_YW_LZ_HW_HOC_IMANUM}, the HOC model is defined by
\begin{align}\label{hoc energy}
	\mathcal{E}^{\rm hoc}(u)&=\int_{\Omega}\sum_{\rho\in\mathscr{R}}\phi_{\rho}\left(\rho\nabla u+\frac{\rho^3}{24}\nabla^3 u\right) \,{\rm d}x
	\nonumber \\ 
	&=: \int_{\Omega}W_{\rm hoc}(\nabla u,\nabla^3 u)\,{\rm d}x,
\end{align}
where $W_{\rm hoc}$ denotes the energy density of higher-order continuum models. 
It is worthwhile noting that the CB model can be obtained by including only the first-order term in \eqref{hoc energy} so that
\begin{eqnarray}\label{eq:cb}
	\mathcal{E}^{\rm cb}(u):=\int_{\Omega}W_{\rm cb}(\nabla u)\,{\rm d}x, \quad \text{where}~~ W_{\rm cb}(\nabla u) = \sum_{\rho\in\mathscr{R}}\phi_{\rho}(\rho\nabla u).
\end{eqnarray}

Since the Euler-Lagrange equation of the HOC model is a sixth-order elliptic equation~\cite{2021_YW_LZ_HW_HOC_IMANUM}, the displacement space $u$ in $\mathcal{E}^{\rm hoc}(\cdot)$ and $W_{\rm hoc}(\nabla \cdot,\nabla^3 \cdot)$ should belong to an appropriate higher-order Sobolev space to ensure well-posedness and regularity, i.e.,
\begin{align}
	\label{equ: space of BQHOCE without coarse graining}
	\mathcal{U}^{\rm hoc}:=\big\{ u\in C^2(\Omega;\mathbb{R})~\big|&~u|_{T_i}\in \mathbb{P}_5(T_i),\,\nabla^j u(x+2N)=\nabla^ju(x),\,\nonumber\\
	&u(0)=0,i=-N+1,...,N,\,j=0,1,2.\big\}.
\end{align}

For future reference, we can always associate a function $u\in\mathscr{U}^{1,2}$ with its smooth interpolation $\Pi_{\rm hoc} u \in \mathcal{U}^{\rm hoc}$ by an interpolation operator $\Pi_{\rm hoc}: \mathscr{U}^{1,2} \rightarrow \mathcal{U}^{\rm hoc}$~(cf.~\cite[Equation (2.3)]{2021_YW_LZ_HW_HOC_IMANUM}).

For a given dead load $f$, we define the inner product as
$$
(f,u)_\Omega := \int_{\Omega} f \cdot u \, {\rm d} x.
$$
The HOC problem is then formulated as the following local minimization problem, where the objective is to find a displacement field that minimizes the higher-order continuum energy functional:
\begin{align}
	\label{minimization of HOC}
	u^{\rm hoc} \in \arg \min \left\{ \mathcal{E}^{\rm hoc}(u) - (f,u)_\Omega \ \big|\ u \in \mathcal{U}^{\rm hoc} \right\}.
\end{align}

As shown in \cite[Theorem 6.3]{2021_YW_LZ_HW_HOC_IMANUM}, compared with the CB model, the HOC model achieves fourth-order accuracy with respect to the interatomic spacing $\varepsilon$. To be precise, we present the following error estimates for the HOC model. 
\begin{equation}
	\label{eq: accuracy_HOC}
	\Vert \nabla \Pi_{\rm hoc} u^{\rm a}-\nabla u^{\rm hoc}\Vert_{L^2(\Omega)} \leq C \err^{\hoc}_{\Omega}(\Pi_{\rm hoc} u^{\rm a}),
\end{equation}
where
$$
\mathrm{err}^{\rm hoc}_{\Omega}(\Pi_{\rm hoc} u^{\rm a}) := \Vert\nabla^5\Pi_{\rm hoc} u^{\rm a}\Vert_{L^2(\Omega)} + \Vert\nabla^2\Pi_{\rm hoc} u^{\rm a}\nabla^4\Pi_{\rm hoc} u^{\rm a}\Vert_{L^2(\Omega)} + \Vert\nabla^3\Pi_{\rm hoc} u^{\rm a}(\nabla^2\Pi_{\rm hoc} u^{\rm a})^2\Vert_{L^2(\Omega)}
$$
$$
+ \Vert\nabla^3\Pi_{\rm hoc} u^{\rm a}\Vert^2_{L^4(\Omega)}\Vert\nabla^2\Pi_{\rm hoc} u^{\rm a}\Vert^4_{L^8(\Omega)} + \Vert\nabla^3\Pi_{\rm hoc} f\Vert_{L^2(\Omega)},
$$
and $C$ is a generic constant depending on $M^{(j,4)},\,j=1,\ldots,5$.

It is also shown in~\cite[Equation (6.9)]{2021_YW_LZ_HW_HOC_IMANUM} that if we scale the system with the interatomic spacing $\varepsilon:= \frac{1}{2N}$ such that 
$$
X:=\varepsilon x,\, U:=\varepsilon u ~~ \text{and}~~  F := \varepsilon^{-1}f,$$ and consequently  
\begin{align}
	\Vert\nabla^i f\Vert_{L^2(\Omega)} &= \varepsilon^{i+1-\frac{1}{2}} \Vert\nabla^i F\Vert_{L^2}, \quad~ i = 0, 1, 2, 3, \nonumber \\
	\Vert\nabla^j u\Vert_{L^2(\Omega)} &= \varepsilon^{j-1-\frac{1}{2}} \Vert\nabla^j U\Vert_{L^2}, \quad j = 1, 2, 3, 4, 5.\nonumber
\end{align}
With a proper modification of the interpolation operator $\Pi_{\rm hoc}$, we are able to obtain the following result:
\[
\Vert\nabla \Pi_{\rm hoc} U^{\rm a}-\nabla U^{\rm hoc}\Vert_{L^2} \leq C \varepsilon^4,
\]
which yields the resulting higher-order continuum model is fourth-order accuracy to the reference atomistic model with respect to lattice spacing.


\section{Blended Coupling of Atomistic and High-order Continuum Models}
\label{sec: B-QHOCF and B-QHOCE}

We note that the estimate in \eqref{eq: accuracy_HOC} only holds when the external force $f$ is sufficiently smooth. This corresponds to the situation that the system is not defected. If $f$ possesses certain singularity, which represents defects in one dimension, we need to couple the atomistic model with the continuum model together to provide the resolution. Such scheme is often termed as the atomistic-to-continuum (a/c) coupling methods \cite{2013_ML_CO_AC_Coupling_ACTANUM,2012_CO_LZ_GRAC_Construction_SIAMNUM,2011_BK_ML_BQCE_1D_SIMNUM,2016_XL_CO_AS_BK_BQC_Anal_2D_NUMMATH,2016_EV_CO_AS_Boundary_Conditions_for_Crystal_Lattice_ARMA}. In this work, we concentrate on the so called blended a/c methods \cite{2011_BK_ML_BQCE_1D_SIMNUM,2016_XL_CO_AS_BK_BQC_Anal_2D_NUMMATH,2012_XL_ML_CO_BQCF_Stab_1D_2D_MMS} which are formulated in detail in the rest of this section. 

\subsection{Formulation of atomistic-to-high-order-continuum coupling}
\label{sec:sub: Formula}
We first decompose the computational domain $\Omega$ into the atomistic, the blending and the continuum regions which are denoted by $\Omega_{\rm a}, \Omega_{\rm b}$ and $\Omega_{\rm c}$ respectively. The lattice $\Lambda$ is correspondingly decomposed as 
$$
\Lambda_{\rm a}:= \Lambda\cap\Omega_{\rm a}, \quad \Lambda_{\rm b}:= \Lambda\cap\Omega_{\rm b} \quad \text{and} \quad  \Lambda_{\rm c}:= \Lambda\cap\Omega_{\rm c}
.$$  We then define the blending function $\beta\in C^{2,1}: \Omega\to[0,1]$ such that $\beta = 0$ in $\Omega_{\rm a}$, $\beta = 1$ in $\Omega_{\rm c}$ and ${\rm supp}(\nabla\beta) = \Omega_{\rm b}$. We often assume that the singularity of $f$ is contained in $\Omega_{\rm a}$. 

We define the continuous piecewise affine interpolation $Qv \in \mathbb{P}_1(\T)$ for any $v: \Omega \rightarrow \R$. Given an external load $f$ with possible singularity in $\Omega_a$, the classic energy-based blended quasicontinuum (B-QCE) method is to find
\begin{align}
	\label{eq: BQCE}
	u^{\rm bqce} &\in\arg \min\left\{\mathcal{E}^{\rm bqce}(u)-\langle f,u\rangle_\Lambda~\big|~u\in \mathscr{U}^{1,2} \right\}, 
\end{align}
where $$\mathcal{E}^{\rm bqce}(u):= \sum_{\xi\in\Lambda_{\rm a}\cup\Lambda_{\rm b}}\big(1-\beta(\xi)\big)
\sum_{\rho\in\mathscr{R}}\phi_{\rho}\big(D_{\rho}u(\xi)\big)+\int_{\Omega_b\cup\Omega_c}Q\beta W_{\rm cb}(\nabla u)\,{\rm d}x,$$
with the Cauchy-Born energy density defined by \eqref{eq:cb}.

Instead of searching for the local energy minimizers, the classic force-based quasicontinuum (B-QCF) method looks for equilibriums by solving the force-balance equations in the following weak form
\begin{equation}
	\label{eq: BQCF}
	\langle\mathcal{F}^{\rm bqcf}(u),v\rangle=\langle f, v\rangle_\Lambda,\;\forall v\in\mathscr{U}^{1,2}.
\end{equation}
where 
\begin{equation*}
	\langle\mathcal{F}^{\rm bqcf}(u),v\rangle:=\langle\delta\mathcal{E}^{\rm a}(u),(1-\beta)v\rangle+\langle\delta\mathcal{E}^{\rm cb}(u),Q(\beta v)\rangle.
\end{equation*}

The B-QCE and the B-QCF methods are formulated in \cite{2011_BK_ML_BQCE_1D_SIMNUM, 2013_ML_CO_BK_BQCE_CMAME, 2012_XL_ML_CO_BQCF_Stab_1D_2D_MMS, 2012_JL_PM_Convergence_BQCF_3D_No_Defects_CPAM} and analyzed in detail in \cite{2013_ML_CO_AC_Coupling_ACTANUM, 2016_XL_CO_AS_BK_BQC_Anal_2D_NUMMATH,2012_JL_PM_Convergence_BQCF_3D_No_Defects_CPAM,2014_HL_ML_CO_AS_BQCF_Compt_CMAME}. By introducing higher-order continuum model in $\Omega_{\rm b}\cup\Omega_{\rm c}$, we formulate the energy-based blended higher order quasicontinuum (B-QHOCE) method and the force-based blended higher order quasicontinuum (B-QHOCF) method in one dimension.

To make the formulation rigorous, we first introduce the space of displacements for the B-QHOC methods. Since the third-order derivative of the displacement $u$ is included in the higher continuum model, we employ piecewise quintic polynomials in the blending and the continuum regions. The space of the displacement of the BQHOC methods is then formally defined by
\begin{align}
	\mathcal{U}^{\rm bqhoc}:=&\bigg\{u\in C(\Omega;\mathbb{R})\cap C^2(\Omega_{\rm b}\cup\Omega_{\rm c};\mathbb{R})~\bigg|~u|_{T_i}\in \mathbb{P}_1(T_i),\text{for}\; T_i\subset\Omega_{\rm a}, \;u|_{T_i}\in \mathbb{P}_5(T_i),\nonumber\\
	&\text{for}\,T_i\subset\Omega_{\rm b}\cup\Omega_{\rm c},u(0)=0,\;\nabla^j u(x)=\nabla^j u(x+2N),\;i=-N+1,..,N,\;j=0,1,2.\bigg\}.\nonumber
\end{align}

The B-QHOCE method aims to find
\begin{align}
	&\quad u^{\rm bqhoce}\in \arg \min\left\{\mathcal{E}^{\rm bqhoce}(u)-(\langle f,u\rangle)\ \big| \ u\in\mathcal{U}^{\rm bqhoc}\right\},\label{eq:BQHOCE}
\end{align}
where the internal energy is given by
\begin{equation}
	\mathcal{E}^{\rm bqhoce}(u) := \sum\nolimits_{\xi\in\Lambda_{\rm a}\cup\Lambda_{\rm b}}\big(1-\beta(\xi)\big)
	\sum\nolimits_{\rho\in\mathscr{R}}\phi_{\rho}\big(D_{\rho}u(\xi)\big)+\int_{\Omega_{\rm b}\cup\Omega_{\rm c}}\beta W_{\rm hoc}(\nabla u,\nabla^3 u)\,{\rm d}x,\nonumber
\end{equation}
with the energy density of the higher-order continuum model \( W_{\rm hoc} \) defined in \eqref{hoc energy}. The external energy is defined as
\begin{align}
	(\langle f,u\rangle) := \langle f,u\rangle_{\Lambda_{\rm a}}+(f, u)_{\Omega_{\rm b} \cup \Omega_{\rm c}}.
	\label{BQHOCE_external_force}
\end{align}

The B-QHOCF method seeks to solve the force-balance equation in the following weak form:
\begin{align}
	\langle \mathcal{F}^{\rm bqhocf}(u), v \rangle = \langle f, (1-\beta)u \rangle_{\Lambda_{\rm a} \cup \Lambda_{\rm b}} + (f, \beta u)_{\Omega_{\rm b} \cup \Omega_{\rm c}}, \quad \forall v \in \mathcal{U}^{\rm bqhoc}, \label{eq:B-QHOCF}
\end{align}
where
\begin{equation*}
	\langle \mathcal{F}^{\rm bqhocf}(u), v \rangle := \langle \delta \mathcal{E}^a(u), (1-\beta)v \rangle + \langle \delta \mathcal{E}^{\rm hoc}(u), \beta u \rangle.
\end{equation*}

In the following section, we present the main results of this work, specifically the {\it a priori} error estimate for the B-QHOCE method defined in~\eqref{eq:BQHOCE}, along with its rigorous proof. Additionally, we formally state the results for the B-QHOCF method defined in~\eqref{eq:B-QHOCF}, while deferring the detailed rigorous proof to a forthcoming work.


\subsection{Main results}
\label{sec:sub:results}
We first derive the error estimates for the B-QHOC method. For completeness, we also state the corresponding results for the B-QCE and B-QCF methods. While these results have been established in previous analytical work~\cite{2016_XL_CO_AS_BK_BQC_Anal_2D_NUMMATH}, we adapt them to fit the framework of our current setting.

The key observation here is that: under certain assumptions on the stability of the methods, the B-QHOCF solutions have substantial improved accuracy compared to the B-QCE and the B-QCF solutions, whereas the BQHOCE solutions have the same level of accuracy as the B-QCE solutions. To be more clear, we observe that the B-QHOCF method produces solutions with the same level of accuracy as those of the HOC method, but for defected systems rather than the nondefected ones. On the other hand, the B-QHOCE method has little improvement on the accuracy even though the HOC model is utilized in the blending and continuum regions. 

To present the error estimates,  we define the interpolation operator $\Pi:\mathscr{U}^{1,2}\to \mathcal{U}^{\rm bqhoc}$ by
\begin{align}
	\label{equ: Pi}
	\Pi u := \left\{
	\begin{aligned}
		\Pi_{\rm hoc} u, \quad &\text{for } x \in \Omega_{\rm b} \cup \Omega_{\rm c}, \\
		Q u, \quad &\text{for } x \in \Omega_{\rm a}.
	\end{aligned}
	\right.
\end{align}
such that $\Pi u$ is piecewise affine in $\Omega_{\rm a}$ and piecewise quintic in $\Omega_{\rm b}\cup\Omega_{\rm c}$. 

It is shown in~\cite{2011_BK_ML_BQCE_1D_SIMNUM,2012_XL_ML_CO_BQCF_Stab_1D_2D_MMS} that the following stability conditions are rigorously established:
\begin{align}
	\inf_{v\in\mathscr{U}^{1,2}}\langle\delta^2\mathcal{E}^{\bqce}(u^{\rm a} )v,v\rangle > 0, \quad\quad\quad
	\inf_{v\in\mathscr{U}^{1,2}}\langle\delta\mathcal{F}^{\rm bqcf}(u^{\rm a})v,v\rangle > 0.
	\label{stability of B-QCE and B-QHOCF}
\end{align}

For B-QHOCE and B-QHOCF, it is known from \cite{2021_YW_LZ_HW_HOC_IMANUM} that the higher-order Cauchy-Born model satisfies the stability condition. Based on the stability of B-QCE and B-QCF \eqref{stability of B-QCE and B-QHOCF}, it is reasonable to assume that B-QHOCE and B-QHOCF also satisfy the stability condition. A rigorous mathematical proof of this conclusion will be provided in a upcoming work. The stability conditions for B-QHOCE and B-QHOCF can be expressed as follows:
\begin{align}
	\inf_{v\in\mathcal{U}^{\rm bqhoc}}\langle\delta^2\mathcal{E}^{\bqhoce}(\Pi u^{\rm a} )v,v\rangle > 0, \quad
	\inf_{v\in\mathcal{U}^{\rm bqhoc}}\langle\delta\mathcal{F}^{\rm bqhocf}(\Pi u^{\rm a})v,v\rangle > 0.
	\label{stability of B-QHOCE and B-QHOCF}
\end{align}

Before introducing Theorem \ref{a priori of bqhoce and bqhocf}, we first provide a brief review of the existing results. Specifically, the error estimates in \eqref{a priori error of B-QCE} and \eqref{a priori error of B-QCF} correspond to the previously established B-QCE and B-QCF methods \cite{2013_ML_CO_AC_Coupling_ACTANUM, 2016_XL_CO_AS_BK_BQC_Anal_2D_NUMMATH}. These results are included here for comparison with the subsequent findings, serving as a foundation for the error estimates of the blended atomistic-higher-order continuum methods.

%
\begin{theorem}\label{a priori of bqhoce and bqhocf}
	\textbf{({ A priori} error estimate)} Suppose $u^{\rm a}$ is a stable atomistic solution of \eqref{Atomistic variational problem} in the sense of \eqref{atomistic model stability}, and we assume $u^{\rm a}$ and $\Pi u^{\rm a}$ satisfies the stability of \eqref{stability of B-QCE and B-QHOCF} and \eqref{stability of B-QHOCE and B-QHOCF} , then there exists a stable solution $u^{\rm bqce},u^{\rm bqcf}\in\mathscr{U}^{1,2}$ and $u^{\rm bqhoce}$ of problems \eqref{eq: BQCE},\eqref{eq: BQCF},\eqref{eq:BQHOCE},\eqref{eq:B-QHOCF}, respectively, such that
	\begin{align}
		&\Vert\nabla  u^{\rm a}-\nabla u^{\rm bqce}\Vert_{L^2(\Omega)}\leq C(\Vert\nabla^2\beta\Vert_{L^2(\Omega_{\rm b})}+\Vert\nabla\beta\nabla^2  u^{\rm a}\Vert_{L^2(\Omega_{\rm b})}+\mathrm{err}^{\rm cb}_{\bar{\Omega}_{\rm c}}( u^{\rm a})+\Vert\nabla f\Vert_{L^2(\bar{\Omega}_{\rm c})}),\label{a priori error of B-QCE}\\
		&\Vert\nabla  u^{\rm a}-\nabla u^{\rm bqcf}\Vert_{L^2(\Omega)}\leq C(\mathrm{err}^{\rm cb}_{\bar{\Omega}_{\rm c}}( u^{\rm a})+\Vert\nabla f\Vert_{L^2(\bar{\Omega}_{\rm c})}),\label{a priori error of B-QCF}\\
		&\Vert\nabla  \Pi u^{\rm a}-\nabla u^{\rm bqhoce}\Vert_{L^2(\Omega_{\rm b})}\leq C(\Vert\nabla^2\beta\Vert_{L^2(\Omega_{\rm b})}+\Vert\nabla\beta\nabla^2 \Pi u^{\rm a}\Vert_{L^2(\Omega_{\rm b})}+\mathrm{err}^{\rm hoc}_{\bar{\Omega}_{\rm c}}(\Pi u^{\rm a})+\Vert\nabla^3f\Vert_{L^2(\bar{\Omega}_{\rm c})}),\label{a priori error of B-QHOCE}
	\end{align}
	where  $\bar{\Omega}_{\rm c}=\Omega_{\rm b}\cup\Omega_{\rm c}+B_{2r_{\rm cut}+1}$ and $C$ depends on $M^{(j,4)},\,j=2,...,5$.
\end{theorem}

\begin{remark}[Error estimate for B-QHOCF method]
	It is worth noting that we can anticipate the error estimate for the B-QHOCF method to be
	\begin{equation}
		\Vert\nabla \Pi u^{\rm a}-\nabla u^{\rm bqhocf}\Vert_{L^2(\Omega)} \leq C(\mathrm{err}^{\rm hoc}_{\bar{\Omega}{\rm c}}(\Pi u^{\rm a})+\Vert\nabla^3f\Vert_{L^2(\bar{\Omega}_{\rm c})}).\label{a priori error of B-QHOCF}
	\end{equation}
	This suggests that the B-QHOCF method significantly improves accuracy compared to B-QCE and B-QCF, whereas B-QHOCE achieves only the same accuracy level as B-QCE. We leave the detailed proof and rigorous theorem for the B-QHOCF method to a forthcoming paper, while in this work, we provide numerical validation of this estimate in Section~\ref{sec: Numerical experiments}. The primary objective of this study is to establish a formulation for atomistic-to-higher-order continuum models, explore the feasibility of coupling atomistic models with higher-order continuum models to enhance the accuracy of classical a/c coupling methods, and rigorously analyze the energy-based approach.
\end{remark}

\begin{remark}[Scaling]
	From \cite{2013_ML_CO_AC_Coupling_ACTANUM}, we can obtain $\Vert\nabla^2\beta\Vert_{L^2(\Omega)}\lesssim|\Omega_{\rm b}|^{-3/2}$  and $\Vert\nabla\beta\Vert_{L^\infty(\Omega)}\lesssim|\Omega_b|^{-1}$, where $|\Omega_{\rm b}|$ is the length of $\Omega_b$. In order to determine the convergence order with respect to $\varepsilon$, we let $|\Omega|\leq C|\Omega_{\rm b}|$, where $C$ is a non-zero constant, then $ |\Omega_{\rm b}|^{-1}\leq 1/C \varepsilon$. Thus, $\Vert\nabla^2\beta\Vert_{L^2(\Omega)}+\Vert\nabla\beta\nabla^2 \Pi u^{\rm a}\Vert_{L^2(\Omega)}\lesssim \varepsilon^{2} .$
	Combining the scaling of $\mathrm{err}^{\rm hoc}_{\bar{\Omega}_{\rm c}}(\Pi u^{\rm a})$ and $\mathrm{err}^{\rm cb}_{\bar{\Omega}_{\rm c}}(u^{\rm a})$ in Section \ref{sec: A-C Models}, we summarize the convergence orders of the four methods with respect to $\varepsilon$ in Table \ref{tab: Convergence Rates Details}.

	\begin{table}[H]
		\caption{Comparison of Convergence Rates for Errors}
		\begin{center}
			\begin{tabular}{|c|c|c|c|}
				\hline
				& Coupling Error & Modeling Error & Total Error \\
				\hline
				B-QCE   & $\varepsilon^2$ & $\varepsilon^2$ & $\varepsilon^2$ \\
				B-QHOCE & $\varepsilon^2$ & $\varepsilon^4$ & $\varepsilon^2$ \\
				B-QCF   & - & $\varepsilon^2$ & $\varepsilon^2$ \\
				B-QHOCF & - & $\varepsilon^4$ & $\varepsilon^4$ \\
				\hline
			\end{tabular}
		\end{center}
		\label{tab: Convergence Rates Details}
	\end{table}

\end{remark}

\begin{remark}[Modeling error bound]
	B-QHOCE improves accuracy by incorporating higher-order Cauchy-Born model, but its ability to further enhance accuracy is limited due to inherent interpolation errors and residual terms at the interface. The interpolation errors come from the blending function of the atomistic model, which cannot be improved within the classical B-QCE framework. Additionally, after considering the higher-order Cauchy-Born model, a residual term appears at the atomistic-continuum interface. This term arises due to the non-simply connected region of the atomistic model, and as a result, it cannot be eliminated as long as the atomistic-continuum coupling framework is used. For a detailed analysis, please refer to Section~\ref{sec:sub:Why Can Not B-QHOCE Improve the Accuracy?}.
\end{remark}
\section{{\it A Priori} Analysis}
\label{sec: consistency}

In this section, we provide a rigorous proof of the {\it a priori} error estimate for the B-QHOCE method, explaining why incorporating the HOC model does not necessarily improve the accuracy of the energy-based blended a/c method. While this result may seem counterintuitive, we offer a mathematical justification to clarify the underlying reasons. We begin in Section~\ref{sec:sub: pointwise estimate} by establishing a pointwise estimate of the stress error. Building on this analysis, Section~\ref{sec:sub:Why Can Not B-QHOCE Improve the Accuracy?} examines the limitations of B-QHOCE in enhancing accuracy. We then derive the consistency error estimate in Section~\ref{sec:sub:consistency err}, followed by the final {\it a priori} error estimate in Section~\ref{sec:sub: priori err}.

\subsection{Pointwise estimate of the error in stress}
\label{sec:sub: pointwise estimate}

To derive the pointwise error estimate for stress, we first consider the first variations of the various energy functionals. Specifically, the first variation of $\mathcal{E}^{\rm a}$ is given by:
\begin{equation}
	\left\langle\delta\mathcal{E}^{\rm a}(u),v\right\rangle=\sum_{\xi\in\Lambda}\sum_{\rho\in\mathscr{R}}\phi'_{\rho}\big(D_{\rho}u(\xi)\big)D_{\rho}v(\xi), 
	\label{first order variation of atomistic model}
\end{equation}
and the first variation of $\mathcal{E}^{\rm bqhoce}$ is given by
\begin{align}
	\langle \delta\mathcal{E}^{\rm bqhoce}(u),v\rangle &= \sum_{\xi\in\Lambda}\big(1-\beta(\xi)\big)
	\sum_{\rho\in\mathscr{R}}\phi'_{\rho}\big(D_{\rho}u(\xi)\big)D_{\rho}v(\xi)\nonumber\\
	&+\sum_{\rho\in\mathscr{R}}\int_{\Omega}\beta \phi_{\rho}'\left(\rho\nabla u+\frac{\rho^3}{24}\nabla^3u\right)\left(\rho\nabla v+\frac{\rho^3}{24}\nabla^3v\right)\,{\rm d}x.\label{first order variation of bqhoce}
\end{align}

Given a lattice function $v$, we define the nodal interpolation of $v$ by 
\begin{eqnarray}\label{eq:vhat}
	\hat{v}:=\sum_{\xi\in\mathbb{Z}}v(\xi)\zeta(x-\xi), \quad \forall \xi\in\mathbb{Z},
\end{eqnarray}
where $\zeta\in W^{3,\infty}$ and $\zeta(\cdot-\xi)$ is a basis function associated with the lattice site $\xi$ \cite[Section 2.2]{2021_YW_LZ_HW_HOC_IMANUM}. To convert $\delta\mathcal{E}^{\rm a}$ into a form that is local in $\nabla v$, we introduce an alternative interpolation based on the convolution of the nodal basis interpolation $\hat{v}$ with $\zeta$: 
$$
\tilde{v}(x):=(\zeta*\hat{v})(x)=\int_{\Omega}\zeta(x-y)\hat{v}(y)\,{\rm d}y.
$$ 
Hence, we have 
$$
D_{\rho}\tilde{v}(\xi)=\int_{0}^{1}\nabla_{\rho}\tilde{v}(\xi+t\rho)\,{\rm d}t=\int_{\mathbb{R}}\int_{0}^{1}\zeta(\xi+t\rho-x)\,{\rm d}t\cdot\nabla_{\rho}\hat{v}\,{\rm d}x:=\int_{\mathbb{R}}\chi_{\xi,\rho}(x)\cdot\nabla_{\rho}\hat{v}\,{\rm d}x,
$$
where 
$
\chi_{\xi,\rho}(x)=\int_{0}^{1}\zeta(\xi+t\rho-x)\,{\rm d}t.
$ 
We then replace the test function $v$ in 
\eqref{first order variation of atomistic model} by $\tilde{v}$ to obtain 
\begin{align}
	\left\langle\delta\mathcal{E}^{\rm a}(u),\tilde{v}\right\rangle&=\sum_{\xi\in\Lambda}\sum_{\rho\in\mathscr{R}}\phi'_{\rho}\big(D_{\rho}u(\xi)\big)\int_{\mathbb{R}}\chi_{\xi,\rho}(x)\cdot\nabla_{\rho}\hat{v}\,{\rm d}x\nonumber\\
	&=\int_{\mathbb{R}}\bigg[\sum_{\xi\in\Lambda}\sum_{\rho\in\mathscr{R}}\rho\phi_{\rho}'\big(D_{\rho}u(\xi)\big)\chi_{\xi,\rho}(x)\bigg]\nabla\hat{v}\,{\rm d}x\nonumber\\
	&=\int_{\Omega}\bigg[\sum_{\xi\in\mathbb{Z}}\sum_{\rho\in\mathscr{R}}\rho\phi_{\rho}'\big(D_{\rho}u(\xi)\big)\chi_{\xi,\rho}(x)\bigg]\nabla\hat{v}\,{\rm d}x\nonumber\\
	&=:\int_{\Omega}S^{\rm a}(u;x)\cdot\nabla\hat{v}\,{\rm d}x, \nonumber
\end{align}
where the atomistic stress tensor is given by 
$$
S^{\rm a}(u;x)=\sum_{\xi\in\mathbb{Z}}\sum_{\rho\in\mathscr{R}}\rho\phi'_\rho\big(D_\rho u(\xi)\big)\chi_{\xi,\rho}(x).
$$

Given \( v \in \mathcal{U}^{\rm bqhoc} \), we define \( v^{\rm a} = \Pi' v \) where \( \Pi': \mathcal{U}^{\rm bqhoc} \to \mathscr{U}^{1,2} \) satisfies  
\begin{align}
	\label{equ: Pi'}
	\widetilde{(\Pi' v)}(\xi) &= v(\xi), \qquad \quad ~ \text{for } \xi \in \Lambda^{\rm a}, \nonumber \\
	\Pi'v(\xi) &= (\zeta * v)(\xi), \quad \text{for } \xi \in \Lambda^{\rm b} \cup \Lambda^{\rm c}.
\end{align}
The stress of the higher-order continuum model is given by
\begin{align} S^{\rm hoc}(u; x) &= \sum_{\rho \in \mathscr{R}} \Big[\frac{\rho^6}{576}\phi''_{\rho}(\rho\nabla u+\frac{\rho^3}{24}\nabla^3 u)\nabla^5 u  +\frac{\rho^3}{24}\phi'''_{\rho}(\rho\nabla u+\frac{\rho^3}{24}\nabla^3 u)(\rho\nabla^2u+\frac{\rho^3}{24}\nabla^4 u)^2 \nonumber\\
	&+\frac{\rho^4}{24}\phi''_{\rho}(\rho\nabla u+\frac{\rho^3}{24}\nabla^3 u)\nabla^3 u\nonumber+\rho\phi'_{\rho}(\rho\nabla u+\frac{\rho^3}{24}\nabla^3 u)],\nonumber \end{align} 
which is derived from \cite[Equation 4.6]{2021_YW_LZ_HW_HOC_IMANUM}.

Using the properties of $\Pi'$ defined by~\eqref{equ: Pi'}, we can transform the atomistic model part of the error into an integral form. Then through integration by parts and applying the periodic boundary conditions of the test function, we can obtain  
\begin{align}
	\langle \delta\mathcal{E}^{\rm bqhoce}(u), v \rangle - \langle \delta\mathcal{E}^{\rm a}(u), \widetilde{v^{\rm a}} \rangle 
	&= \sum_{\rho \in \mathscr{R}} \int_{\Omega} \beta \phi_{\rho}'\left(\rho\nabla v + \frac{\rho^3}{24}\nabla^3 v\right) \,{\rm d} x - \sum_{\xi \in \Lambda} \beta(\xi) \sum_{\rho \in \mathscr{R}} \phi'_{\rho} D_{\rho}v^{\rm a}(\xi)\nonumber \\
	&= \int_{\Omega} \left[ \beta S^{\rm hoc}(u; x) + K(u; x) - \sum_{\rho,\xi} \beta(\xi) \chi_{\xi,\rho}(x) \rho \phi'_{\rho} \right] \nabla v \, \,{\rm d} x \nonumber\\
	&=: \int_{\Omega} R^\beta(u; x) \, \,{\rm d} x,\nonumber
\end{align}
where 
\begin{equation}
	K(u;x)=\int_{\Omega}-[\frac{\rho^3}{12}\nabla\beta\phi''_{\rho}(\rho\nabla u+\frac{\rho^3}{24}\nabla^3 u)(\rho\nabla^2 u+\frac{\rho^3}{24}\nabla^4 u)+\frac{\rho^3}{24}\nabla^2\beta\phi'_{\rho}(\rho\nabla u+\frac{\rho^3}{24}\nabla^3 u)]{\rm d} x\nonumber
\end{equation}is the residual term introduced at the interface due to the introduction of the higher-order continuum model
and 
\begin{equation}
	\label{equ:total err}
	R^\beta(u; x) = \beta S^{\rm hoc}(u; x) + K(u; x) - \sum_{\rho \in \mathscr{R}} \sum_{\xi \in \mathbb{Z}} \beta(\xi) \chi_{\xi,\rho}(x) \rho \phi'_{\rho}\big(D_{\rho}u(\xi)\big)
\end{equation}
is the total residual term.

\begin{lemma}\label{lem: |R(u;x|}
	Let $u\in W^{5,\infty},x\in\Omega$, then
	\begin{align}
		|R^\beta(u;x)|\leq&~ C(\Vert\nabla^2\beta\Vert_{L^\infty(\nu_x)}+\Vert\nabla\beta\nabla^2 u\Vert_{L^\infty(\nu_x)}+\Vert\nabla^5u\Vert_{L^\infty(\nu_x)}+\Vert\nabla^2u\nabla^4u\Vert_{L^\infty(\nu_x)}\nonumber\\
		&+\Vert\nabla^3u(\nabla^2u)^2\Vert_{L^\infty(\nu_x)}+\Vert\nabla^3u\Vert^2_{{L^\infty(\nu_x)}}+\Vert\nabla^2u\Vert^4_{L^\infty(\nu_x)}),
	\end{align}
	where $C$ depends on $M^{(j,4)},j=2,...,5$ and $\nu_x:= B_{2r_{\rm cut}+1}(x)$ is the neighborhood of some $x\in\mathbb{R}$.
\end{lemma} 

\begin{proof}
	Firstly, the second term in \eqref{equ:total err} can be estimated by the inverse estimates for polynomials \cite{2002_Ciarlet_FEM}[The. 3.2.6] that
	$$
	|K(u;x)|\lesssim M^{(1,2)}\Vert\nabla^2\beta\Vert_{L^{\infty}(\nu_x)}+M^{(2,4)}\Vert\nabla\beta\nabla^2u\Vert_{L^{\infty}(\nu_x)}\nonumber.$$ 
	Then we expand  $$\beta(\xi)=\beta+\nabla\beta\cdot(\xi-x)+O(\delta_2)$$ and $\phi_{\rho}'(D_{\rho}u)$  in $\nabla_{\rho}u(x)$
	, where $\delta_2:=\Vert\nabla^2\beta\Vert_{L^{\infty}(\nu_x)}$, which yields
	\begin{align}
		T(u;x):=&~\beta S^{\rm hoc}-\sum_{\rho\in\mathscr{R}}\sum_{\xi\in\mathbb{Z}}\beta(\xi)\chi_{\xi,\rho}(x)\rho\phi'_{\rho}\big(D_{\rho}u(\xi)\big) \nonumber\\
		=&~\beta S^{\rm hoc}-\sum_{\rho\in\mathscr{R}}\sum_{\xi\in\mathbb{Z}}[\beta+\nabla\beta\cdot(\xi-x)]\big(\rho\phi'_\rho(D_\rho u)\big)\chi_{\xi,\rho}+O(\delta_2)\nonumber\\
		=&~\beta[S^{\rm hoc}-S^{\rm a}]-\nabla\beta\sum_{\rho\in\mathscr{R}}\rho\sum_{\xi\in\mathbb{Z}}\phi_{\rho}'\big(D_{\rho}u(\xi)\big)(\xi-x)\chi_{\xi,\rho}+O(\delta_2)\nonumber\\
		=&~\beta[S^{\rm hoc}-S^{\rm a}]-\nabla\beta\sum_{\rho\in\mathscr{R}}\rho\sum_{\xi\in\mathbb{Z}}\phi_{\rho}'(\rho\nabla u)(\xi-x)\chi_{\xi,\rho}\nonumber\\
		&-\nabla\beta\sum_{\rho\in\mathscr{R}}\rho\sum_{\xi\in\mathbb{Z}}\phi_{\rho}''(\mu_2 )(D_\rho u(\xi)-\rho\nabla u)(\xi-x)\chi_{\xi,\rho}+O(\delta_2)\nonumber\\
		=:&~T_1-T_2-T_3+O(\delta_2),\label{equ: R(u;x) 2th term}
	\end{align}
	where $\mu_2\in {\rm conv}\left\{\nabla_{\rho}u,D_\rho u(\xi)\right\}$. For the second term of \eqref{equ: R(u;x) 2th term}, applying the equation 
	$$\sum_{\xi\in\mathbb{Z}}\chi_{\xi,\rho}(x)(\xi-x)^k=\frac{(-\rho)^k}{k+1},\,k=0,1,2,3$$ from \cite[Lemma 4.1]{2021_YW_LZ_HW_HOC_IMANUM}, we can obtain:
	\begin{align}
		\label{equ: Symmetry of 1st}
		T_2=\nabla\beta\sum_{\xi\in\mathbb{Z}}(\xi-x)\chi_{\xi,\rho}(x)\frac{1}{2}\sum_{\rho\in\pm\mathscr{R}}|\rho|\phi_{\rho}'(|\rho|\nabla u)=-\frac{1}{4}\nabla\beta\sum_{\rho\in\pm\mathscr{R}}\rho|\rho|\phi_{\rho}'(|\rho|\nabla u)=0.
	\end{align}
	Inserting the expansion 
	$$
	D_\rho u(\xi)=\rho\nabla u(x)+\nabla^2u(x)\rho(\xi-x+\frac{1}{2}\rho)+O(\epsilon_3)
	$$  for the third term of \eqref{equ: R(u;x) 2th term}, where $\epsilon_j = \Vert\nabla^j u\Vert_{L^{\infty}(\nu_x)}$ , as well as applying the same equation as mentioned above, one can deduce
	$$|T_3|=\left|\nabla\beta\nabla^2 u\sum_{\rho\in\mathscr{R}}\rho\phi_{\rho}''(\rho\nabla u )\sum_{\xi\in\mathbb{Z}}(\xi-x+\frac{1}{2}\rho)(\xi-x)\chi_{\xi,\rho}\right|\leq M^{(2,4)}\Vert\nabla\beta\nabla^2 u\Vert_{L^\infty(\nu_x)}.$$
	Using similar techniques as above, an error estimate for $T_1$ can be obtained, 
	$$
	|T_1|\lesssim(\Vert\nabla^5u\Vert_{L^\infty(\nu_x)}+\Vert\nabla^2u\nabla^4u\Vert_{L^\infty(\nu_x)}+\Vert\nabla^3u(\nabla^2u)^2\Vert_{L^\infty(\nu_x)}+\Vert\nabla^3u\Vert^2_{{L^\infty(\nu_x)}}+\Vert\nabla^2u\Vert^4_{L^\infty(\nu_x)}),
	$$
	and the detailed proof provided in \cite[Lemma 4.2]{2021_YW_LZ_HW_HOC_IMANUM}. Combining the error estimates for $T_1, T_2,T_3$ , we obtain the stated results. 
\end{proof}

	
	\subsection{Discussion}
	\label{sec:sub:Why Can Not B-QHOCE Improve the Accuracy?}

	B-QHOCE is based on BQCE and introduces the higher-order Cauchy-Born model in the continuum region with the expectation of achieving higher accuracy. However, as shown in the lemma \ref{lem: |R(u;x|}, its dominant error terms remain the same as those of BQCE. Next, we will analyze the sources of these dominant error terms and discuss why it is not possible to further improve accuracy within this framework.
	
	As shown in~\eqref{equ:total err}, we decompose point-wise error of stress $R^\beta(u; x)$ into two terms, $T(u;x)$ and $K(u;x)$, i.e., 
	\begin{align}
		R^\beta(u; x) &= \beta S^{\rm hoc}(u; x) - \sum_{\rho \in \mathscr{R}} \sum_{\xi \in \mathbb{Z}} \beta(\xi) \chi_{\xi,\rho}(x) \rho \phi'_{\rho}\big(D_{\rho}u(\xi)\big) + K(u; x) \nonumber\\
		&=:T(u;x)+K(u;x).
	\end{align}
	As established in Theorem~\ref{a priori of bqhoce and bqhocf}, certain error terms persist, specifically:  
	
	\begin{itemize}
		\item \textit{Interpolation error of $\beta$:} This arises due to the discrete nature of the atomistic model, where the blending function $\beta$ is interpolated between lattice points, introducing unavoidable discretization errors.
		\item \textit{Residual term $K(u;x)$:} The introduction of the higher-order Cauchy-Born model at the interface generates an additional residual term, which does not vanish within the current coupling framework.
	\end{itemize}
	
	In the following, we analyze these two terms in detail.
	
	\subsubsection{Blended ghost force}
	\label{sec:sub:sub: beta}
	We define $I_0$ as the piecewise constant interpolation, then for the atomistic model, the blending function $\beta$ takes values only at the lattice points, which can be understood as a piecewise constant interpolation. Therefore, we can decompose the error into the stress error $E_{\rm stress} $ and the interpolation error  
	$E_{\rm interp}$, and we have
	\begin{align}
		T(u;x)&=\beta(x) S^{\rm hoc}(u;x)-\sum_{\rho\in\mathscr{R}}\sum_{\xi\in\mathbb{Z}}\beta(\xi)\chi_{\xi,\rho}(x)\rho\phi'_{\rho}\big(D_{\rho}u(\xi)\big),\nonumber\\
		&=\beta(x) S^{\rm hoc}(u;x)-\sum_{\rho\in\mathscr{R}}\sum_{\xi\in\mathbb{Z}}I_0\beta(x)\chi_{\xi,\rho}(x)\rho\phi'_{\rho}\big(D_{\rho}u(\xi)\big)\nonumber\\
		&=\beta(x) (S^{\rm hoc}- S^{\rm a})+\sum_{\rho\in\mathscr{R}}\sum_{\xi\in\mathbb{Z}}\big(\beta(x)-I_0\beta(x)\big)\chi_{\xi,\rho}(x)\rho\phi'_{\rho}\big(D_{\rho}u(\xi)\big) \\ \nonumber &=: E_{\text{stress}}+E_{\text{interp}}.\nonumber
	\end{align}  
	Here $E_{\text{stress}}$  represents the error in stress between the higher-order continuum model and the atomistic model, which achieves improved accuracy due to the use of the higher-order continuum model. Meanwhile, $E_{\text{interp}}$ is the estimate of the interpolation error for the blending function $\beta$ at non-lattice points, the first-order term vanishes due to the point symmetry in the interaction potential \eqref{equ: Symmetry of 1st}, improving the interpolation error estimate from first-order to second-order.
	
	Therefore, $ E_{\text{interp}} $ is the dominant error term of $T(u;x) $, and its origin is essentially due to ghost forces. Improving the estimation of the interpolation error for $ \beta $ within the classic B-QCE framework is not feasible, as the atomistic model can only take values at the lattice points in the blending region, which corresponds to the piecewise constant interpolation $I_0$. Therefore, $\Vert\nabla^2\beta\Vert_{L^2}$   is difficult to improve further. However, we can refer to the approach in the blended ghost force correction (BGFC) method~\cite{2016_CO_LZ_GForce_Removal_SISC}, where an additional ghost force correction term is introduced to optimize the constant in front of $\Vert\nabla^2\beta\Vert_{L^2}$, thereby improving the overall accuracy. 
	

	\subsubsection{Residual term K(u;x)}
	\label{sec:sub:sub: K(u;x)}
	
	It is worthwhile noting that $K(u;x)$ does not exist in the classic BQCE. It is essentially generated at the a/c interface when the test function $\nabla ^3v$ is converted to $\nabla v$ by the process of integration by parts. Due to the more intuitive nature of this process in higher dimensions, we explain it from a higher-dimensional perspective. Additionally, this error term persists in higher dimensions, where the coupling between the atomistic and continuum regions generates residuals at the interface, which cannot be eliminated. Therefore, this term must be particularly considered when dealing with high-dimensional coupling.
	
	We present the difference of the process of the integration by parts for HOC and B-QHOCE in a continuum setting:
	\begin{align}
		\text{HOC}:\qquad\int_{\Omega} f\,\nabla^3 v\,{\rm d}V&=\oint_{\partial \Omega} f\,  \nabla^2v\,{\rm d} s-\int_{\Omega}  \nabla f \, \nabla^2v\,{\rm d}V\label{hoc integration by parts}\\
		\text{B-QHOCE}:\int_{\Omega_{\rm b}\cup\Omega_{\rm c}} f\,\nabla^3 v\,{\rm d}V&=\oint_{\partial \Omega\cup\partial\Omega_{\rm a}} f\,  \nabla^2v\, {\rm d} s-\int_{\Omega_{\rm b}\cup\Omega_{\rm c}}  \nabla f \, \nabla^2v\,{\rm d}V\label{B-QHOCE integration by parts}
	\end{align}
	For the pure higher-order continuum model, $K(u;x)$ does not appear because the function space can be controlled to make the line integral on the right-hand side of \eqref{hoc integration by parts} equal to 0 at $\partial\Omega$.
	
	When we consider the special case of \( \beta \), where it takes the characteristic function  (1 in \( \Omega_{\rm c} \cup \Omega_{\rm b} \) and 0 in \( \Omega_{\rm a} \)), a blended interface is formed between the atomistic and continuum regions. In this case, a new error term is generated at the boundary of \( \Omega_{\rm a} \) in \eqref{B-QHOCE integration by parts}, and this error term cannot be eliminated. This residual cannot be removed because we cannot predefine the boundary conditions at the atomistic-continuum interface. Thus the residual \( K(u;x) \) persists and cannot be eliminated in the context of a/c coupling methods.
	
	\subsubsection{Summary of Error Sources and Limitations}
	\label{sec:sub:sub: summary Limitations}
	Through the analysis in Sections \ref{sec:sub:sub: beta} and \ref{sec:sub:sub: K(u;x)}, it is clear that for the error term $ T(u;x) $, although the accuracy of the interpolation error estimate for $ \beta $ cannot be directly improved, it can be enhanced by applying the ghost force correction method. As for the term $ K(u;x) $, it is a new error introduced after the incorporation of the higher-order Cauchy-Born  model. This error arises because, after using atomistic-continuum coupling methods, the continuum region is no longer a simply connected domain, leading to this error at the interface. Therefore, this is the reason why B-QHOCE cannot improve the accuracy further.
	
	On the other hand, for the B-QHOCF method, the coupling is achieved by blending forces, which means directly blending the weak form. As a result, there is no ghost force error\cite{2016_XL_CO_AS_BK_BQC_Anal_2D_NUMMATH}, and no error is introduced from the process of converting $\nabla^3v$ to $\nabla v$. Therefore, the accuracy of the continuum model in the continuum region directly determines the overall accuracy of the coupling method, thus allowing B-QHOCF to achieve higher accuracy.

\subsection{Consistency error estimate of the B-QHOCE method}
	\label{sec:sub:consistency err}

In this section, we provide the consistency error estimate of the proposed B-QHOCE method.

\begin{theorem}\label{the : consistency of B-QHOCE}
	\textbf{(Consistency error of B-QHOCE)} Let $u\in\mathscr{U}^{1,2}$ and $f,v\in\mathcal{U}^{\rm bqhoc}$, then we can obtain
	\begin{align}
		\langle \delta\mathcal{E}^{\rm bqhoce}(\Pi u),v\rangle-\langle\delta\mathcal{E}^{\rm a}(u),\widetilde{v^{\rm a}}\rangle\leq& ~C(\Vert\nabla^2\beta\Vert_{L^2(\Omega_{\rm b})}+\Vert\nabla\beta\nabla^2 \Pi u\Vert_{L^2(\Omega_{\rm b})}+\Vert\nabla^5\Pi u\Vert_{L^2(\bar{\Omega}_{\rm c})}\nonumber\\
		&+\Vert\nabla^2\Pi u\nabla^4\Pi u\Vert_{L^2(\bar{\Omega}_{\rm c})}+\Vert\nabla^3\Pi u(\nabla^2\Pi u)^2\Vert_{L^2(\bar{\Omega}_{\rm c})}\nonumber\\
		&+\Vert\nabla^3\Pi u\Vert^2_{L^4(\bar{\Omega}_{\rm c})}\Vert\nabla^2\Pi u\Vert^4_{L^8(\bar{\Omega}_{\rm c})})\Vert\nabla v\Vert_{L^2(\Omega)},\label{consistency of internal energy}
	\end{align}
	where $\bar{\Omega}_{\rm c}=\Omega_{\rm b}\cup\Omega_{\rm c}+B_{2r_{\rm cut}+1}$ and $C$ depends on $M^{(j,4)},j=2,...5$. 
\end{theorem}

\begin{proof}
	From the definition of $\Pi$ by \eqref{equ: Pi} we observe that $\Pi u|_{\Lambda}=u|_{\Lambda}$. From \eqref{equ:total err}, we can obtain
	\begin{align}
		\langle\delta\mathcal{E}^{\rm bqhoce}(\Pi u) ,v\rangle-\langle\delta\mathcal{E}^{\rm a}(u),\widetilde{v^{\rm a}}\rangle=\langle\delta\mathcal{E}^{\rm bqhoce}(\Pi u), v\rangle-\langle\delta\mathcal{E}^{\rm a}(\Pi u),\widetilde{v^{\rm a}}\rangle=\int_{\Omega}R^\beta(\Pi u;x)\cdot\nabla v \,{\rm d}  x.\nonumber
	\end{align}
	Applying the Cauchy-Schwarz inequality yields
	$$\int_{\Omega}R^\beta(\Pi u;x)\cdot\nabla v {\rm d} \, x\leq\Vert R^\beta(\Pi u;x)\Vert_{L^2(\Omega)}\cdot\Vert\nabla v\Vert_{L^2(\Omega)}$$
	By using local norm-equivalence for polynomials with $\beta$ replaced with $\hat{\beta}$ and Lemma \ref{lem: |R(u;x|}, we can obtain
	\begin{align}
		|R^{\hat{\beta}}(\Pi u;x)|^2\lesssim&~ \Vert\nabla^2\hat{\beta}\Vert^2_{L^\infty(\nu_x)}+\Vert\nabla\hat{\beta}\nabla^2 \Pi u\Vert^2_{L^\infty(\nu_x)}+\Vert\nabla^5\Pi u\Vert^2_{L^\infty(\nu_x)}+\Vert\nabla^2\Pi u\nabla^4\Pi u\Vert^2_{L^\infty(\nu_x)}\nonumber\\
		&+\Vert\nabla^3\Pi u(\nabla^2\Pi u)^2\Vert^2_{L^\infty(\nu_x)}+\Vert\nabla^3\Pi u\Vert^4_{{L^\infty(\nu_x)}}+\Vert\nabla^2\Pi u\Vert^8_{L^\infty(\nu_x)}\nonumber\\
		\lesssim&~ \Vert\nabla^2\hat{\beta}\Vert^2_{L^2(\nu_x)}+\Vert\nabla\hat{\beta}\nabla^2 \Pi u\Vert^2_{L^2(\nu_x)}+\Vert\nabla^5\Pi u\Vert^2_{L^2(\nu_x)}+\Vert\nabla^2\Pi u\nabla^4\Pi u\Vert^2_{L^2(\nu_x)}\nonumber\\
		&+\Vert\nabla^3\Pi u(\nabla^2\Pi u)^2\Vert^2_{L^2(\nu_x)}+\Vert\nabla^3\Pi u\Vert^4_{{L^4(\nu_x)}}+\Vert\nabla^2\Pi u\Vert^8_{L^8(\nu_x)}.\label{equ:R(Pi u;x)}
	\end{align}
	Using an interpolation error estimate for $\beta$, we can obtain
	\begin{align}
		\Vert R^{\beta}(\Pi u;x)-R^{\hat{\beta}}(\Pi u;x)\Vert_{L^2(\Omega)}=\Vert\beta S^{\rm hoc}-\hat{\beta}S^{\rm hoc}\Vert_{L^2(\Omega)}\lesssim \Vert \nabla^2\beta\Vert_{L^2(\Omega_{\rm b})}.\label{eq: estimate of R beta(u;x)}
	\end{align}
	Integrating \eqref{equ:R(Pi u;x)} over $\Omega$ and the estimate of \eqref{eq: estimate of R beta(u;x)}, we have 
	\begin{align}
		\Vert R^{{\beta}}(\Pi u;x)\Vert_{L^2(\Omega)}
		\lesssim~ &\Vert\nabla^2\beta\Vert_{L^2(\Omega_{\rm b})}+\Vert\nabla\beta\nabla^2 \Pi u\Vert_{L^2(\Omega_{\rm b})}+\Vert\nabla^5\Pi u\Vert_{L^2(\bar{\Omega}_{\rm c})}+\Vert\nabla^2\Pi u\nabla^4\Pi u\Vert_{L^2(\bar{\Omega}_{\rm c})}\nonumber\\
		&+\Vert\nabla^3\Pi u(\nabla^2\Pi u)^2\Vert_{L^2(\bar{\Omega}_{\rm c})}+\Vert\nabla^3\Pi u\Vert^2_{L^4(\bar{\Omega}_{\rm c})}\Vert\nabla^2\Pi u\Vert^4_{L^8(\bar{\Omega}_{\rm c})},
	\end{align}
	which yields the stated result.
\end{proof}

\begin{lemma}
	\label{lem: consistency of external energy}
	Suppose $f\in W^{3,\infty}(\Omega,\mathbb{R})$, and let $\langle f,v\rangle_{\Lambda}$ and  $(\langle f,v^{\rm a}\rangle)$ be defined in \eqref{Atomistic external force} and \eqref{BQHOCE_external_force}, respectively. We have the following estimate:
	\begin{equation}
		\big|( \langle f,v\rangle)-\langle f,\widetilde{v^{\rm a}}\rangle_{\Lambda}\big|\lesssim\Vert\nabla^3f\Vert_{L^2(\bar{\Omega}_{\rm c})}\Vert\nabla v\Vert_{L^2(\Omega)}. \nonumber
	\end{equation}
\end{lemma}
\begin{proof}
	According to the $v^a$ defined by \eqref{equ: Pi'}, we have
	\begin{align}
		(\langle f,v\rangle)-\langle f,\tilde{v^{\rm a}}\rangle_{\Lambda}&=\langle f,v\rangle_{\Lambda_{\rm a}}+(f, v)_{\Omega_{\rm b} \cup \Omega_{\rm c}}-(\langle f,v\rangle_{\Lambda_{\rm a}}+\langle f,\widetilde{v^{\rm a}}\rangle_{\Lambda_{\rm b}\cup\Lambda_{\rm c}} )\nonumber\\
		&=(f, v)_{\Omega_{\rm b} \cup \Omega_{\rm c}}-\langle f,\widetilde{v^{\rm a}}\rangle_{\Lambda_{\rm b}\cup\Lambda_{\rm c}}\nonumber\\
		&=\int_{\Omega_{\rm b}\cup\Omega_{\rm c}}v(x)\cdot (f(x)-\sum_{\xi\in\Lambda}\zeta(x-\xi)f(\xi))\,{\rm d} x \nonumber\\
		&:=\int_{\Omega_{\rm b}\cup\Omega_{\rm c}}v(x)\cdot g(x)\,{\rm d}x
	\end{align}
	where $g(x)=f(x)-\sum_{\xi\in\Lambda}\zeta(x-\xi)f(\xi)$.
	
	Applying the Cauchy-Schwarz inequality and the \textit{Poincaré} inequality we obtain the estimate 
	\begin{align}
		|(\langle f,v\rangle)-\langle f,\tilde{v^{\rm a}}\rangle_{\Lambda}|&\leq\Vert g\Vert_{L^2(\Omega_{\rm b}\cup\Omega_{\rm c})}\Vert v\Vert_{L^2(\Omega_{\rm b}\cup\Omega_{\rm c})}\nonumber\\
		&\leq \Vert g\Vert_{L^2(\Omega_{\rm b}\cup\Omega_{\rm c})}\Vert \nabla v\Vert_{L^2(\Omega)}.\label{external estimate 1}
	\end{align}
	According to the Bramble-Hilbert lemma we can estimate the $L^2$-norm of $g$ by 
	\begin{align}
		\Vert g\Vert_{L^2(\Omega_{\rm b}\cup\Omega_{\rm c})}\lesssim \Vert \nabla^3 f\Vert_{L^2(\bar{\Omega}_{\rm c})}. \label{external estimate 2}
	\end{align}
	The combination of \eqref{external estimate 1} and \eqref{external estimate 2} leads to the required result.
\end{proof}

\subsection{{\it A priori} error estimate}
\label{sec:sub: priori err}

We first state the well-known inverse function theorem (cf.~\cite[Lemma 2.2]{2011_CO_1D_QNL_MATHCOMP}), which is the key of proving the {\it a priori} error estimate of the proposed B-QHOCE method.

\begin{lemma}\label{Inverse function theorem}
	\textbf{(Inverse function theorem)} Let $\mathcal{A},\mathcal{B}$ be Banach spaces, $\mathcal{A}_s$ a subset of $\mathcal{A}$, and let $\mathcal{F}:\mathcal{A}_s\to\mathcal{B}$ be Frechet differentiable with Lipschitz-continuous derivative $\delta \mathcal{F}$:
	\begin{align}
		\Vert\delta\mathcal{F}(U)-\delta\mathcal{F}(V)\Vert_{L(\mathcal{B},\;\mathcal{A})}\leq M\Vert U-V\Vert_{\mathcal{A}}, \qquad\forall U,V\in\mathcal{A},\nonumber
	\end{align}
	where $M$ is a Lipschitz constant. Let $X\in\mathcal{A}_s$ and suppose also that there exists $\eta,\sigma>0$ such that
	\begin{eqnarray}
		\Vert\mathcal{F}(X)\Vert_{\mathcal{B}}\leq\eta, \quad\Vert\delta\mathcal{F}(X)^{-1}\Vert_{L(\mathcal{B},\;\mathcal{A})}\leq\sigma, \quad 2M\eta\sigma^2<1.\label{sigma,eta,leq1}
	\end{eqnarray}
	Then there exists a locally unique $Y\in\mathcal{A}$ such that $\mathcal{F}(Y)=0$ and $\Vert Y-X\Vert_{\mathcal{A}}\leq2\eta\sigma$.
\end{lemma}

In the following proof, we will apply the  inverse function theorem (Lemma \ref{Inverse function theorem}) to establish the {\it a priori} error estimate \eqref{a priori error of B-QHOCE} for B-QHOCE in Theorem~\ref{a priori of bqhoce and bqhocf}.
\begin{proof}[Proof of \eqref{a priori error of B-QHOCE} in Theorem \ref{a priori of bqhoce and bqhocf}]
	\quad Following the framework of the prior error estimates in \cite{2021_YW_LZ_HW_HOC_IMANUM,2016_XL_CO_AS_BK_BQC_Anal_2D_NUMMATH}, we divide the proof into three steps. We define the operator $\mathcal{F}:\mathcal{U}^{\rm bqhoc}\to W^{1,2}$ by 
	$$\langle\mathcal{F}(u),v\rangle:=\langle\delta\mathcal{E}^{\rm bqhoce}(u),v\rangle-(\langle f,v\rangle).$$
	Due to the Lipschitz continuity of the potential $\Phi_{\rho}$~\cite{2013_ML_CO_AC_Coupling_ACTANUM} and the blending function $\beta$, $\delta\mathcal{F}$ is also Lipschitz continuous.
	
	\textbf{Step 1:Consistency.} 
	We shall require $\mathcal{F}(X)$ to be consistent. If we put the interpolation of the  atomistic solution $u^{\rm a}$ in $\mathcal{E}^{\rm bqhoce}$, we have
	\begin{align}
		\langle\mathcal{F}(X),v\rangle&=\langle\delta\mathcal{E}^{\rm bqhoce}(\Pi u),v\rangle-(\langle f,v\rangle)\nonumber\\
		&=\langle\delta\mathcal{E}^{\rm bqhoce}(\Pi u),v\rangle-(\langle f,v\rangle)-\langle\delta\mathcal{E}^a(u^{\rm a}),v^a\rangle+\langle f,v^a\rangle_{\Lambda}\nonumber\\
		&=\langle\delta\mathcal{E}^{\rm bqhoce}(\Pi u),v\rangle-\langle\delta\mathcal{E}^a(u^{\rm a}),v^a\rangle+\langle f,v^a\rangle_{\Lambda}-(\langle f,v\rangle)\nonumber\\
		&=:T_1+T_2,\nonumber
	\end{align}
	where $v^a$ is defined from \eqref{equ: Pi'}, and we decompose consistency into two parts; $T_1$ is the consistency error of internal energy and $T_2$ is the consistency error of the external external energy. 
	
	From Theorem \ref{the : consistency of B-QHOCE} and Lemma \ref{lem: consistency of external energy}, we can obtain
	\begin{align}
		\Vert\mathcal{F}(X)\Vert_{W^{1,2}}\lesssim& C(\Vert\nabla^2\beta\Vert_{L^2(\Omega_{\rm b})}+\Vert\nabla\beta\nabla^2 \Pi u\Vert_{L^2(\Omega_{\rm b})}+\Vert\nabla^5\Pi u\Vert_{L^2(\bar{\Omega}_{\rm c})}+\Vert\nabla^2\Pi u\nabla^4\Pi u\Vert_{L^2(\bar{\Omega}_{\rm c})}\nonumber\\
		&+\Vert\nabla^3\Pi u(\nabla^2\Pi u)^2\Vert_{L^2(\bar{\Omega}_{\rm c})}+\Vert\nabla^3\Pi u\Vert^2_{L^4(\bar{\Omega}_{\rm c})}\Vert\nabla^2\Pi u\Vert^4_{L^8(\bar{\Omega}_{\rm c})}+\Vert\nabla^3f\Vert_{L^2(\bar{\Omega}_{\rm c})}):=\eta\label{converge_rate_BQHOCF}
	\end{align}
	where $C$ depends on   $M^{(j,4)},\,j=2,...,5$ .
	
	\textbf{Step 2: Stability. }
	From  assumption \eqref{stability of B-QHOCE and B-QHOCF}, we can deduce that  
	\begin{align}
		\langle\delta\mathcal{E}^{\rm bqhoce}(\Pi u^{\rm a})v,v\rangle\geq c_0\Vert\nabla v\Vert^2_{L^2},\nonumber
	\end{align} 
	where $c_0>0$, Then
	$$\Vert\delta\mathcal{F}(X)^{-1}\Vert_{L^2(W^{1,2},\,\mathcal{U}^{\rm bqhoc})}\leq(c_{0})^{-1}=:\sigma$$
	\textbf{Step 3 : Inverse function theorem} We assume that $\eta_1,\eta_2$ are sufficiently small such that $\Vert\nabla^j\Pi u^{\rm a}\Vert_{L^2(\bar{\Omega}_{\rm c})}\leq \eta_1,j=2,3,4,5$ and $\Vert\nabla^3 f\Vert_{L^2(\bar{\Omega}_{\rm c})}\leq\eta_2$. Combining the consistency result in Step 1 and the stability result in Step 2 and applying the inverse function theorem shown in Lemma~\ref{Inverse function theorem}, we obtain the existence of the solution of B-QHOCE in $W^{1,2}$ and the error estimate
	$$\Vert\nabla\Pi u^{\rm a}-\nabla u^{\rm bqhoce}\Vert_{L^2(\Omega)}\leq\frac{2\eta}{\sigma},$$
	which can be guaranteed if we choose $\eta_1,\eta_2$ to be sufficiently small.
\end{proof}

\subsection{Coarse-graining}
\label{sub:sub: coarse-graining}
Considering that coarse-graining involves changing the function space \eqref{equ: space of BQHOCE without coarse graining} to the coarse graining space, which introduces many new terms that require error estimation, the process becomes overly cumbersome. Therefore, we will present a prior error estimate for the B-QHOCE method that includes coarse-graining without proof, and validate it through numerical experiments in Section \ref{sec:sub:Coarsening error of Numerical}.

we consider coarse-graining continuum region, we partition the domain by choosing a set of representative atoms 
\begin{align}
	\mathcal{L}_{\rm rep}=\left\{\xi_1,\xi_2...,\xi_{N_{\rm rep}}\right\}\subset\Lambda,\nonumber
\end{align}
such that $N_{\rm rep}\ll N$ and $\xi_1<\xi_2<...<\xi_{N_{\rm rep}}$. We assume that $\Lambda_a\cup\Lambda_b\subset\mathcal{L}_{\rm rep}$, and the grid is periodically extended, that is, we define $t_{i+N_{\rm rep}} = t_i + N$. For simplicity, we also assume that $\xi_0 =-N+1$ and $\xi_{N_{\rm rep}} = N$. The mesh size functions for the elements is 
$$h_i = \xi_{i}-\xi_{i-1}, \qquad\text{for} \;i\in\mathbb{Z}.$$
We let $T^h_i=[\xi_i,\xi_{i+1}],\,i=1,2,...,N_{\rm rep}-1$. Then, we define the coarse graining displacement space as 
\begin{align}
	&\mathcal{U}_h^{\rm bqhoc}:=\bigg\{u\in C^0(\Omega,\mathbb{R})\cap C^2(\Omega_{\rm b}\cup\Omega_{\rm c},\mathbb{R})\bigg|u|_{T_i^h}\in \mathbb{P}_1(T_i^h),\,\text{for}\, T_i^h\subset\Omega_{\rm a}, \;u|_{T_i^h}\in \mathbb{P}_5(T),\,\nonumber\\
	&\text{for}\,T_i^h\subset\Omega_{\rm b}\cup\Omega_{\rm c},u(0)=0,\;\nabla^j u(x)=\nabla^j u(x+2N),\;i=1,2,...,N_{\rm rep},\;j=0,1,2.\bigg\}.\nonumber
\end{align}

The B-QHOCE method with coarse-graining aims to find the solution to the following problem 
\begin{equation}
	\quad u_h^{\rm bqhoce}\in \arg \min\left\{\mathcal{E}^{\rm bqhoce}(u)-(\langle f,u\rangle)\ \big| \ u\in\mathcal{U}_h^{\rm bqhoc}\right\},\label{eq:BQHOCE coarse}
\end{equation}
From equation \eqref{eq:BQHOCE}, the error estimate in the B-QHOCE, which does not consider coarse graining, can be categorized into three parts: 
\begin{align}
	&\mathscr{E}^{\rm int}(u):=\Vert\nabla^2\beta\Vert_{L^2(\Omega_{\rm b})}+\Vert\nabla\beta\nabla^2  u\Vert_{L^2(\Omega_{\rm b})},\\
	&\mathscr{E}^{\rm model}(u):=\mathrm{err}^{\rm hoc}_{\bar{\Omega}_{\rm c}}(u),  
\end{align} which represent the coupling error, modeling error.

The coarse-grained interpolation operator is defined as follows:  
\begin{align}  
	\Pi_hu:=\Pi'_{\rm hoc}\tilde{u}  
\end{align}  
where $\Pi'_{\rm hoc}$ denotes the quintic spline interpolation operator on the mesh $T_h$, and $\tilde{u}$ is a $C^6$-continuous Hermite interpolation operator.

Now, we can present the error estimate for the B-QHOCE method that includes coarse graining.

\begin{theorem}
	\label{the: A priori error of coarse grain}
	Suppose $u^{\rm a}$ is a stable atomistic solution of \eqref{Atomistic variational problem} in the sense of \eqref{atomistic model stability}, and we assume that $\Pi_h u^{\rm a}$ satisfies the stability of 
	$$\inf_{v\in\mathcal{U}_h^{\rm bqhoc}}\langle\delta^2\mathcal{E}^{\bqhoce}(\Pi_h u^{\rm a} )v,v\rangle > 0,$$ then there exists a stable solution  $u_h^{\rm bqhoce}$ of problems \eqref{eq:BQHOCE coarse}, such that
	\begin{align}
		\Vert\nabla  \Pi_h u^{\rm a}-\nabla u_h^{\rm bqhoce}\Vert_{L^2(\Omega)}&\leq C\big(\Vert h^5\nabla^6\tilde{u}^{\rm a}\Vert_{L^2(\bar{\Omega}_{\rm c})}+\mathscr{E}^{\rm model}(\tilde{u}^{\rm a})+\mathscr{E}^{\rm int}(\tilde{u}^{\rm a})\big),\label{a priori error of B-QHOCE}
	\end{align}
	where  $\bar{\Omega}_{\rm c}=\Omega_{\rm b}\cup\Omega_{\rm c}+B_{2r_{\rm cut}+1}$ and $C$ depends on $M^{(j,4)},\,j=2,...,6$.
\end{theorem}


\section{Numerical experiments} 
\label{sec: Numerical experiments}

In this section, we validate our theoretical findings through numerical experiments. Our goal is to examine the accuracy of the proposed blended atomistic-to-higher-order continuum (B-QHOCE and B-QHOCF) methods and compare it with the classical B-QCE and B-QCF methods.

We consider a one-dimensional atomistic system subjected to an external force field. The computational domain is fixed as  
\begin{equation}
	F=1, \quad \Omega = \left[-\frac{1}{2},\frac{1}{2}\right].
\end{equation}
The interatomic potential is chosen as a harmonic potential, given by  
\begin{equation}
	\phi(r) = \frac{1}{2} (r - 1)^2,
\end{equation}
which allows for an explicit analysis of the behavior of the coupling methods.

The atomistic system consists of \( N+1 \) atoms, and the discretization parameter is given by \( \epsilon = 1/N \), where \( N \) determines the resolution of the atomistic region. The external force \( f(x) \) is defined as
\begin{equation}
	f(x):=\left\{
	\begin{array}{l}
		-f_{\rm scale}\frac{x+1/2}{x} \quad\;\,\text{for}\; -1/2\leq x<0, \\
		0\quad\quad \quad\quad\quad\;\,\text{for}\; x=0,\\
		f_{\rm scale}\frac{1/2-x}{x}\quad\quad\text{for}\; 0<x\leq1/2.
	\end{array}\right.\nonumber
\end{equation}
This load function is designed to introduce a nontrivial mechanical response in the system, allowing us to assess the performance of different coupling methods in capturing the resulting deformation field.

To ensure that the higher-order Cauchy-Born model is appropriately utilized, we apply \( C^2 \)-continuous finite elements in the blending and continuum regions, following the approach shown in \cite{2021_YW_LZ_HW_HOC_IMANUM}. This choice is motivated by the need for smooth higher-order derivatives in the continuum approximation, which is crucial for maintaining accuracy in the coupling scheme. 

\subsection{Modeling error}
\label{sec:sub:Modeling error of Numerical}
There, we exclude coarse-graining and directly compare the modeling errors between B-QHOCE, B-QHOCF, and the atomistic model, where all atoms are treated as degrees of freedom that need to be computed

We plot the strain $\nabla u^{\rm ac,\epsilon}$ of the solutions in Figure \ref{fig:strain of B-QCE,B-QHOCE} where ${\rm ac} \in \{ \a, \bqce, \bqhoce \}$. We set $f_{\rm scale}=20$ and $N = 600$ so that $\epsilon=1/600$ and we use dashed lines to mark the interfaces of different regions. It is observed that the three methods are barely distinguishable in the atomistic and the continuum regions while jumps occur at the interfaces due to the inconsistency analyzed in Section \ref{sec: consistency}. The curve of the B-QHOCE method is speculated to be closer to that of the atomistic model which implies a slight advantage of the method over the B-QCE.
\begin{figure}[H]
	\centering
	\includegraphics[width=0.6\textwidth]{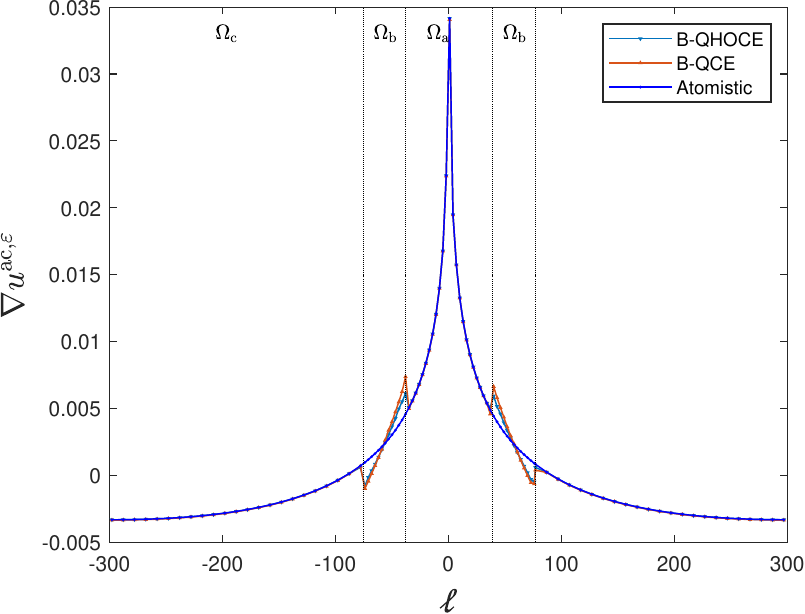}
	\caption{Plot of the strain.}
	\label{fig:strain of B-QCE,B-QHOCE}
\end{figure}

Figure \ref{fig:strain node error} shows the nodal errors in strain of the B-QCE and the B-QHOCE methods. The B-QHOCE method clearly reveals better accuracy than the B-QCE method in the atomistic and the continuum regions. However, the errors of the two methods are comparable in the blending region and are apparently dominant in the total errors.

\begin{figure}[H]
	\centering
	\includegraphics[width=0.6\textwidth]{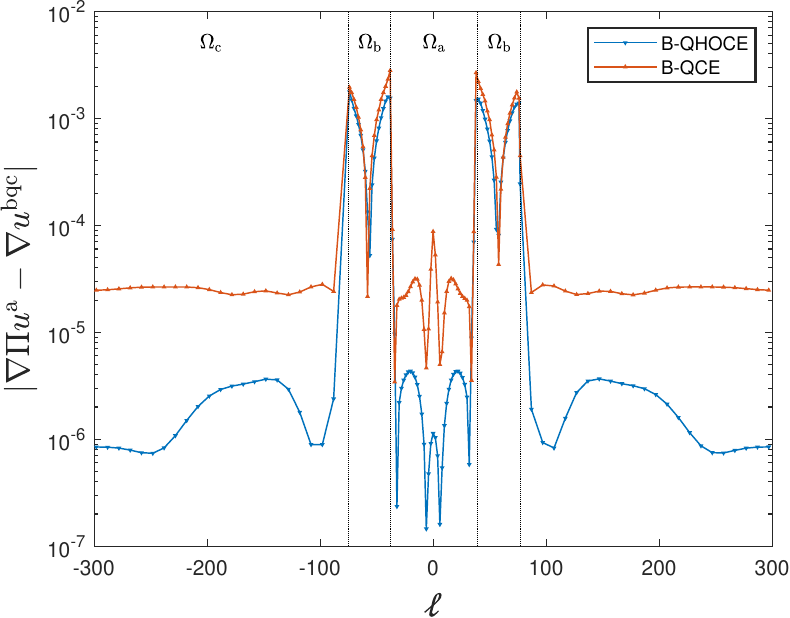}
	\caption{Nodal error in strain.}
	\label{fig:strain node error}
\end{figure}

To emphasize the singular nature of \( f \), we set \( f_{\rm scale} = 20000 \). The convergence rates of the six methods considered in this study are plotted in Figure~\ref{fig:Modeling error}. From the results, we observe that neither the higher-order Cauchy-Born model nor the classical Cauchy-Born model exhibit convergence under singular external forces, with the higher-order model displaying even lower accuracy. This behavior arises because the energy approximation~\eqref{hoc energy} is only valid for smooth displacement fields, underscoring the necessity of integrating the higher-order model within the a/c coupling framework. Furthermore, we note that the B-QHOCE method does not improve the overall rate of convergence, which aligns with our theoretical analysis. In contrast, the B-QHOCF method outperforms all other methods in terms of accuracy, demonstrating its potential advantages. This motivates a more detailed investigation of B-QHOCF, as discussed earlier.
\begin{figure}[H]
	\centering
	\includegraphics[width=0.6\textwidth]{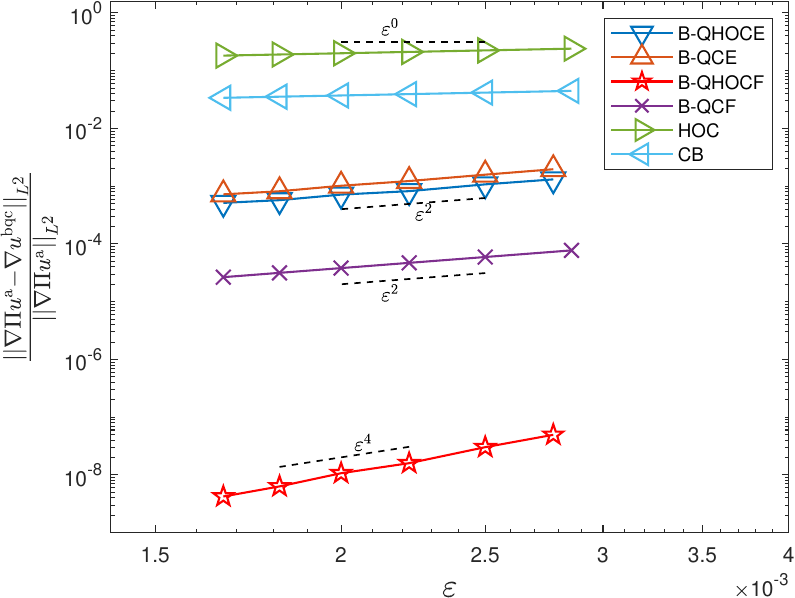}
	\caption{Convergence of relative error with respect to lattice spacing $\varepsilon$.}
	\label{fig:Modeling error}
\end{figure}
\subsection{Coarsening error}
\label{sec:sub:Coarsening error of Numerical}

Although the B-QHOCE method does not significantly improve the overall convergence order with respect to modeling error, it can substantially enhance the accuracy of coarse-graining and effectively reduce the degrees of freedom in the continuum region due to the application of the higher-order Cauchy-Born model and the high-order finite element method. In Theorem \ref{the: A priori error of coarse grain}, we provided the error estimate for the B-QHOCE method, which includes the coarse graining. Next, we validate it through numerical experiments. 

In the experiment, we set N = 100000, with 200 atoms in both the atomistic region and the blending region. A uniform grid is used, and the range of h is from 10,000 to 2,000. Under these conditions, the number of atoms in the continuum region is only between 10 and 50. We compute the $L_2$ error in the continuum region $\Omega_{\rm C}$ between the strain of the atomistic model $\nabla u_{\rm a}$ and the strain of the coarse-grained  B-QHOCE method $\nabla u_h^{\rm bqhoce}$. as shown in Figure ~\ref{fig:corase graining bqhoce}. The results show that the error in the continuum region exhibits fifth-order convergence with respect to $h$, consistent with the theoretical analysis. Thus, it can be concluded that by adopting the higher-order Cauchy-Born model, the number of atoms in the continuum region can be significantly reduced while maintaining the same level of accuracy.
\begin{figure}[H]
	\centering
	\includegraphics[width=0.7\textwidth]{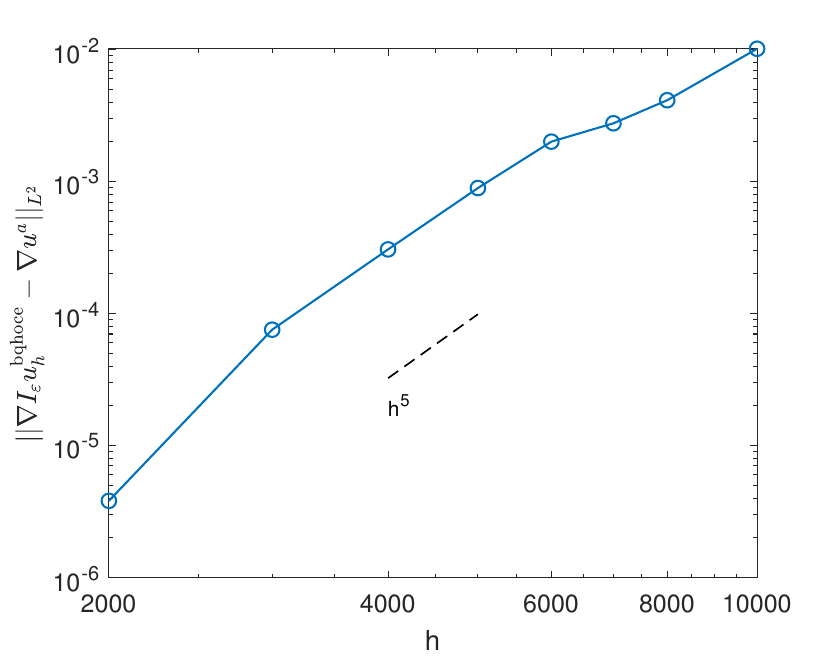}
	\caption{coarsening error}
	\label{fig:corase graining bqhoce}
\end{figure}


\section{Conclusion}
\label{sec:conclusion}

In this work, we introduced higher-order continuum models into blended atomistic-to-continuum (a/c) coupling methods to explore their potential for accuracy improvement. Specifically, we replaced the traditional second-order Cauchy-Born approximation in the continuum region with a higher-order counterpart. Through rigorous error analysis, we demonstrated that while the higher-order model enhances accuracy within the continuum region, it does not improve the overall convergence rate due to new coupling errors arising at the a/c interface. These errors pose a fundamental limitation on achieving higher accuracy and persist in both one- and higher-dimensional settings.

To validate our theoretical findings, we conducted numerical experiments that confirmed the error estimates and demonstrated that incorporating the higher-order continuum model within a blended force-based framework leads to improved accuracy. This study establishes a foundational formulation for atomistic-to-higher-order continuum coupling, assessing its feasibility for enhancing classical a/c methods. The findings provide valuable insights into the role of higher-order continuum approximations in multiscale modeling and highlight potential refinements in coupling strategies.

This work still raise a few open problems which deserve further mathematical analysis.

\begin{itemize}
	
	\item \textit{Stability analysis.} This work does not rigorously establish the stability of the B-QHOCE method, and force-based methods also present challenges. Future work will adapt stability analysis techniques from~\cite{2013_ML_CO_AC_Coupling_ACTANUM,2016_XL_CO_AS_BK_BQC_Anal_2D_NUMMATH,2012_XL_ML_CO_BQCF_Stab_1D_2D_MMS, ortner2023framework, wang2024theoretical} to study both B-QHOCE and B-QHOCF methods, with a particular focus on stability behavior at the interface after incorporating the higher-order continuum model. 
	
	\item \textit{Energy-based atomistic-higher-order continuum coupling with a sharp interface.} The blended approach used in this work introduces coupling errors at the interface that limit overall accuracy. To address this, we plan to explore sharp interface methods~\cite{2012_CO_LZ_GRAC_Construction_SIAMNUM,2006_WE_JL_ZY_GRC_PRB}, reconstructing the interface energy to enable a more precise transition between the atomistic and higher-order continuum models while minimizing residual errors.
	
	\item \textit{Extension to higher dimensions.}  Implementing the higher-order Cauchy-Born model in higher dimensions poses challenges due to the need for third- and higher-order derivatives. We aim to explore spline finite element methods to improve computational efficiency and accuracy, leveraging their ability to represent smooth functions with fewer degrees of freedom.
	
\end{itemize}

Both the theoretical and practical aspects discussed above will be explored in future work.

\bibliographystyle{plain}
\bibliography{MS_Coupling_202502}

\end{document}

%% file: notation.tex
\def\T{\mathcal{T}}

\def\a{\rm a}

\def\p0{\pmb{0}}

\def\R{\mathbb{R}}

\def\<{\langle}